\documentclass{m2an1}
\usepackage{fullpage}
\usepackage{graphicx}
\usepackage[utf8]{inputenc}
\usepackage[english]{babel}

\usepackage{amsmath}
\usepackage{amssymb,amsthm,dsfont,url,appendix,tabularx}
\usepackage{enumitem}

\usepackage{psfrag}
\usepackage{graphicx}
\usepackage{tikz}
\usetikzlibrary{patterns}
\usepackage{pgfplots}

\numberwithin{equation}{section}
\numberwithin{figure}{section}

\newtheorem{theorem}{Theorem}[section]
\newtheorem{lemma}[theorem]{Lemma}

\newtheorem{assumption}[theorem]{Assumption}

\newtheorem{remark}[theorem]{Remark}

\def\mp#1{{\color{black}#1}}

\newcommand{\vertiii}[1]{{\left\vert\kern-0.25ex\left\vert\kern-0.25ex\left\vert #1 
		\right\vert\kern-0.25ex\right\vert\kern-0.25ex\right\vert}}

\newcommand{\R}{\mathbb{R}}

\newcommand{\CE}{\mathcal{E}}

\newcommand{\CA}{\mathcal{A}}

\newcommand{\Bc}{\textbf{c}}

\newcommand{\BV}{\textbf{V}}
\newcommand{\BI}{\textbf{I}}
\newcommand{\BJ}{{\boldsymbol{\mathcal J}}}
\newcommand{\CJ}{{{\mathcal J}}}
\newcommand{\BBI}{{\boldsymbol{\mathcal I}}}
\newcommand{\CI}{{{\mathcal I}}}
\newcommand{\BL}{\textbf{L}}

\newcommand{\BS}{\textbf{S}}
\newcommand{\BX}{\textbf{X}}

\newcommand{\BH}{\textbf{H}}

\newcommand{\Bf}{\textbf{f}}
\newcommand{\Bn}{\textbf{n}}
\newcommand{\Bu}{\textbf{u}}

\newcommand{\Be}{\textbf{e}}
\newcommand{\BP}{\textbf{P}}
\newcommand{\Bv}{\textbf{v}}

\numberwithin{equation}{section}
\numberwithin{figure}{section}
\numberwithin{table}{section}

\newcommand{\abs}[1]{\left\vert #1 \right\vert}
\newcommand{\skp}[1]{\left< #1 \right>}
\newcommand{\norm}[1]{\left\| #1 \right\|}

\newcommand{\T}{\mathcal{T}}

\begin{document}

\title{Analysis of a fully discretized FDM-FEM scheme for solving thermo-elastic-damage coupled nonlinear PDE systems}

\thanks{All authors are funded by Germany Excellence Strategy within the Cluster of Excellence PhoenixD (EXC 2122, Project ID 390833453). 
	Maryam Parvizi is also funded by the Alexander von Humboldt Foundation project named $\mathcal{H}$-matrix approximability of the inverses for FEM, BEM, and FEM-BEM coupling of the electromagnetic problems. }
%
\author{Maryam Parvizi}\address{Leibniz Universit\"{a}t Hanover, Welfengarten 1, 30167 Hannover, Germany~~~ \email{\{parvizi, khodadadian, wick\}@ifam.uni-hannover.de}}
\author{Amirreza Khodadadian}\sameaddress{1}\address{School of Computer Science and Mathematics, Keele University, Keele, UK}
\author{ Thomas Wick}\sameaddress{1} 
\date{\today}
%
%
%
%

\begin{abstract}
In this paper, we consider a nonlinear PDE system governed by a parabolic heat equation coupled in a nonlinear way with a hyperbolic momentum equation describing the behavior of a displacement field coupled with a nonlinear elliptic equation 
based on an internal damage variable.  We present a numerical scheme 
based on a low-order Galerkin finite element method (FEM)  for the space discretization of the time-dependent nonlinear PDE system and an implicit finite difference method (FDM) to discretize in the direction of the time variable. Moreover, we present a priori estimates for the exact and discrete solutions for the pointwise-in-time $L^2$-norm. Based on the a priori estimates, we rigorously prove the convergence of the solutions of the fully discretized system to the exact solutions. Denoting the properties of the internal parameters,  we find the order of convergence concerning the discretization parameters.
\end{abstract}

%

\maketitle

\section{Introduction}
Damage models consist of a system of nonlinear partial differential equations (PDEs) that enable us to monitor and observe the behavior of occurred failure, fracture, and displacements in different materials, especially in brittle, quasi-brittle, and thermoelastic solids \cite{bourdin2000numerical,MR3373458}. In thermoelastic materials, a temperature change (e.g., a thermal shock) leads to a non-uniform volume change and thermal stress. Exceeding the material tensile strength initiates a fracture that can continue until full separation. The material separation 
occurs by mechanical forces or thermoelastic effects. For instance, lasers can be used effectively to cut ceramic substrates and glasses.

In this paper, we consider a nonlinear PDE system, including three coupled equations to model the thermal and mechanical behavior of thermoelastic materials. {This system consists of  a nonlinear hyperbolic momentum equation  coupled with a parabolic heat equation that describe  the behavior of the displacement field  and the heat distribution, respectively. Moreover,  the momentum equation is coupled with a  nonlinear elliptic equation that describes the behavior of an internal variable.}

 To the best of the authors' knowledge, there are only a few papers studying the existence, uniqueness, and regularity properties of the solutions of such nonlinear PDE systems (in the presence of the thermoelastic materials) with respect to the time and space variables in the weak from as well as in the time discretized version.
A model to simulate the thermoelastic fracture problems is presented in  \cite{MR3373458} to describe a phase-field fracture equation coupled with heat conduction.
Another model that studies damage \cite{MR3634026} uses a system of nonlinear PDE system including viscous Cahn–Hilliard equation (to model the phase separation), and momentum balance (to model the displacement) coupled with a thermal system.
A thermodynamic consistent PDE system for phase transition and damage addressing the existence of the weak solution is given in \cite{MR3365562}.
In \cite{MR2596554ex}, the author presents a damage model governed by a PDE system consisting of the momentum equation for the displacement coupled with a heat equation as well as coupled with a rate-independent flow equation for the damage variable in a  strongly nonlinear way.
Concerning thermo-viscoelastic materials, the weak formulation and the existence of solutions for the coupled system are given in \cite{MR3842151,MR2596554ex}.

{In the presence of  material damage or regularized fractures, one approach is to introduce an
internal variable that determines the current state of the process.  In such cases,
  additional nonlinearities, in terms of inequalities, appear in  the  auxiliary equations. Additionally,  the new internal parameters, i.e., $\kappa$ and $\ell$, interact 
with each other in a certain way.
In the discrete setting of the damage models, to guarantee a reliable approximation of the solution, we assume $\kappa=\mathcal{O}(\ell)$ and $h=\mathcal{O}(\ell)$ where 
$h$ is the spatial discretization parameter.  We notice that the relations $\kappa=\mathcal{O}(\ell)$ and $h=\mathcal{O}(\ell)$
go into the direction of $\Gamma$-convergence \cite{ambrosio1990approximation,Braides1998}, but are in general, weaker than the assumptions of the $\Gamma$-convergence theorems.}


 In \cite{MR3373458}, the authors introduce a thermodynamically consistent model  for regularized fracture, and present a finite element method for discretization.

{The main aim of this paper is to  study the stability and convergence of a discrete
scheme based on a low-order Galerkin FEM (for the space discretization), and  an implicit finite
difference scheme (for the time discretization) for solving the time-dependent nonlinear
PDE system. 
In the model described in this paper,
we  allow that  certain internal parameters interact with each other in a certain way.

\noindent\textbf{List of difficulties.}
Below, we list the problems that should be overcome to achieve the main findings of this paper:
\begin{itemize}
	\item 
The internal variable ($\varphi$) may reach zero leading to an elliptic degeneracy in the momentum equation {\cite{rocca2014degenerating}}. To avoid this problem, a parameter $\kappa $  is inserted into the equation, and we need to study the effect of the degenerate limit $\kappa \downarrow 0$ in our analysis, especially in the stability estimates  of the discrete solutions and the order of convergence. 
\item  
Due to the assumptions $\kappa=\mathcal{O}(\ell)$ and $h=\mathcal{O}(\ell)$, we have the same consideration as above for the length scale parameter $\ell$.
\item 
{The highly nonlinear nature of the PDE system, i.e., 
nonlinear quadratic terms appearing in the PDE system,
requires several techniques to study  
convergence of the discrete solutions to the continuous solutions.
\item {The imposed irreversibility condition guarantees that the crack never heals (no crack reverse). This inequality will add the  complexity of the system.}}
\end{itemize}
\textbf{Our results.} { For the time-dependent nonlinear PDE system, 
	we present a discretization scheme based on a low-order Galerkin FEM, and an implicit finite difference scheme to discretize in space and time, respectively.
	The results of this paper can be summarized as follows:
	\begin{itemize}
	\item We present a priori estimates for the exact and discrete solutions  of the momentum equation as well as the heat equation, i.e., estimates for the pointwise-in-time $L^2$-norms of the displacement field, the strain tensor of displacement, the pointwise derivative of the strain tensor, and the heat function (as well as for their discrete counterparts).
	\item Defining the $\CA$-norm as the norm associated with the linear elasticity operator $\CA$, we also provide an a priori estimate for the pointwise-in-time $\CA$-norm of the strain tensor of the displacement field (and the strain tensor of the discrete counterpart of the displacement field).
{	\item For $\tau$ defined as the time discretization parameter, we rigorously
	prove  the convergence of the discrete displacement field,  heat function, and internal variable to their continuous counterparts in the pointwise-in-time $L^2$-norms with the order of convergence $\mathcal{O}(\kappa ^ {-1/2} (\tau+h))$. Additionally, since we are allowed to consider the  relations $\kappa=\mathcal{O}(\ell)$ and $h=\mathcal{O}(\ell)$, we prove that if 
the assumptions  ${\tau }{\ell ^ {-1}}= \mathcal{O}(1)$ and ${\tau }{\kappa ^ {-1}}= \mathcal{O}(1)$ are satisfied, the  convergence  of the discrete solutions to the exact solutions is obtained in the pointwise-in-time $L^2$-norm  with the order of convergence  $\mathcal{O} (\kappa ^ {-1/2}\tau+ (\kappa ^ {-1/2}+\ell ^ {-1/2})h)$. }
	\end{itemize}
\textbf{Outline of the paper.} In Section \ref{sec:mainresults}, we start with fixing some notations and continue with a short introduction to the time-dependent  model PDE system. Section \ref{section3} contains a fully discretized scheme based on 
 a  FDM to discretize in time, and a Galerkin FEM for the spatial discretization. We also provide  a priori estimates for the solutions of the semi-discrete formulation as well as for the solutions of the fully discretized one. 
 Section \ref{section4} is concerned with the main result of this paper, i.e., we investigate the convergence of the solutions of the fully discretized system to the exact solutions. Finally, in Section \ref{numerical}, we present a numerical example to illustrate our theoretical results.

%
%
%

%
%
%
%
%

\section{Problem statement and notation}\label{sec:mainresults}
Let $\Omega \subset \mathbb{R}^d$, $d=2,3$   be   {a bounded and sufficiently regular domain}  with the
boundary $\Gamma := \partial \Omega$.
Through this paper, for $p\ge 1$, we denote  $L^p (\Omega)$ as the usual  Lebesgue spaces on $\Omega$ with the corresponding norm $\norm{\cdot}_{L^p(\Omega)}$. For $p=2$, the space $L^2(\Omega)$ is a Hilbert space with the inner product $\skp{\cdot, \cdot}_{\Omega}$.
For this inner product, when there is no risk of confusion, we drop the subscript $\Omega$.
 Moreover,
the Lebesgue space $L^2(\Gamma)$ is defined  as the space of square integrable functions on $\Gamma$  with the inner product $\skp{\cdot,\cdot}_\Gamma$ and  the corresponding norm $\norm{\cdot}_ {L^2 (\Gamma)}$. 
For $s\ge 0$ and $q \ge 0$, 
 we  use the standard notations for the Sobolev space $W^{s,q}(\Omega)$ with the corresponding norm $\norm{\cdot}_ {s,q,\Omega}$ and semi-norm $\abs{\cdot}_ {s,q,\Omega}$. For the case, $q=2$, we also use the notation 
 $H^s(\Omega)$ with the standard norm $\norm{\cdot}_{H^s(\Omega)}$ and semi-norm $\abs{\cdot}_{H^s(\Omega)}$.
Let $\R ^{d \times d}$ be the space of {$d\times d$} square matrices with  entries in $\R$ and   $\BI \in \R ^{d \times d}$ be the identity matrix. 
We also define the following tensor space
\begin{align*}
[L^2 (\Omega )]^{d\times d}:= \left\lbrace  \boldsymbol{ \tau}  = (\tau _ {i,j} )_ {i,j \in \{1,\cdots,d\}} \quad : \quad \tau _ {i,j} \in L^2 (\Omega) \right\rbrace, 
\end{align*}
with the inner product 
\begin{align*}
\skp{\boldsymbol{ \tau}, \boldsymbol{ \sigma}}:=\int_{\Omega} \boldsymbol{ \tau} : \boldsymbol{ \zeta}\, dx \qquad \forall \boldsymbol{ \tau}, \, \boldsymbol{ \zeta} \in [L^2 (\Omega )]^{d\times d},
\end{align*}
which induces the following norm 
\begin{align*}
\norm{\boldsymbol{\tau}}_{[L^2 (\Omega )]^{d\times d}}=\sqrt{\skp{\boldsymbol{ \tau}, \boldsymbol{ \tau}}} \qquad \forall \boldsymbol{\tau} \in [L^2 (\Omega )]^{d\times d}.
\end{align*}
 Given the tensors $\boldsymbol{ \tau}:= (\tau  _ {ij}) \in [L^2 (\Omega )]^{d\times d}$ and $\boldsymbol{ \sigma}:= (\sigma_ {ij}) \in [L^2 (\Omega )]^{d\times d}$, we consider the following notations for the trace of a tensor and the scalar product of two tensors
\begin{align*}
\operatorname{tr}(\boldsymbol{\boldsymbol{ \tau}}):= \sum_{i=1}^{d} \tau _ {ii}, 
\qquad \boldsymbol{ \tau}: \boldsymbol{ \sigma}: = \sum_{i,j=1}^{d}\tau _ {ij}\sigma _ {ij}.
\end{align*} 
Moreover, the deviatoric
part of a tensor  $\operatorname{\mathbf{dev}}: [L^2 (\Omega )]^{d\times d} \rightarrow [L^2 (\Omega )]^{d\times d} $ is defined as $\operatorname{\mathbf{dev}} (\boldsymbol{ \sigma}) := \boldsymbol{ \sigma} -  1/d \operatorname{tr}(\boldsymbol{ \sigma})\BI$.
Given the Lam\'e parameters $\lambda, \, \mu >0$, the linear mapping $\CA : [L^2 (\Omega )]^{d\times d} \rightarrow [L^2 (\Omega )]^{d\times d}$ is defined as follows
\begin{align}\label{operator-A}
	\CA (\boldsymbol{ \sigma}):= \lambda \operatorname{tr}(\boldsymbol{ \sigma})\BI + 2   \mu \boldsymbol{ \sigma} \qquad \forall \, \boldsymbol{ \sigma} \in  [L^2 (\Omega )]^{d\times d}.
\end{align}
The operator $\CA$ is positive definite and symmetric. Moreover, $\CA ^ {1/2}$ is defined as 
\begin{align*}
	\CA ^ {1/2} (\boldsymbol{ \sigma}) := \sqrt{2\mu} \boldsymbol{ \sigma}+ \frac{\sqrt{2\mu+d\lambda}- \sqrt{2\mu}}{d} \operatorname{tr}(\boldsymbol{ \sigma})\BI \qquad \forall \, \boldsymbol{ \sigma} \in  [L^2 (\Omega )]^{d\times d}.
\end{align*}
With the elasticity tensor $\CA$, we define the following {$\CA$-}norm
$$ \norm{\boldsymbol{\tau}}^2_{\CA}  := \norm{\CA^ {1/2} (\boldsymbol{\tau})}^2_{[L^2 (\Omega )]^{d\times d}} := \int _ \Omega \CA (\boldsymbol{\tau}) : \boldsymbol{\tau}\,dx\qquad\forall \, \boldsymbol{\tau} \in [L^2 (\Omega )]^{d\times d}.$$
Furthermore, we define  the  operator
\begin{align}\label{The nonlinear-elasticity-opt}
\mathcal{B}(\boldsymbol{\tau}):=  {\CA(\boldsymbol{ \tau}):\boldsymbol{ \tau} }\qquad \forall \boldsymbol{\tau} \in  [L^2 (\Omega )]^{d\times d}.
\end{align}
 We continue  with  the definition of the vector space
\begin{align*}
\BL^2 (\Omega):= \left\lbrace \Bv = (v_i)_ {i=1}^ {d} \qquad : \qquad v_i \in L^2 (\Omega)\right\rbrace, 
\end{align*}
with the inner product
$
\skp{\Bu, \Bv}_ {\BL^2 (\Omega)}:= \int_{\Omega} \Bu \cdot \Bv  \,dx
$ 
for all $\Bu, \, \Bv \in \BL^2 (\Omega)$, 
which induces  the norm $\norm{\cdot}_{\BL^2 (\Omega)}$.
When there is no risk of confusion, we use the notation $\norm{\cdot}_{\BL^2(\Omega)}$ to denote both of the norms $\norm{\cdot}_ {[L^2 (\Omega )]^{d\times d}}$ and $\norm{\cdot}_ {\BL^2 (\Omega)}$, and we also drop  ${\BL^2(\Omega)}$ from the subscript $\skp{\cdot, \cdot}_{\BL^2(\Omega)}$.
Moreover, for $s \ge 0$, we set
\begin{align*}
\BH^s (\Omega):= \left\lbrace \Bu = (u_i)_ {i=1}^ {d} \quad : \quad u_i \in H^s (\Omega)\right\rbrace, 
\quad \BH_{0}^1 (\Omega):=\left\lbrace \Bu \in \BH^1 (\Omega)\,:\, \Bu | _ {\Gamma}=0\right\rbrace, 
\end{align*}
with the corresponding norm $\norm{\cdot}_{\BH^s(\Omega)}$ and semi-norm $\abs{\cdot}_{\BH^s(\Omega)}$.
We also denote the strain tensor of displacement by $\CE (\Bu):= \frac{1}{2} \left( \nabla \Bu + \nabla \Bu ^T \right)  $ for $\Bu \in \BL^2 (\Omega ) $. 
Let the operator $\operatorname{\boldsymbol \nabla}\cdot:  [L^2 (\Omega )]^{d\times d} \rightarrow \boldsymbol{H^{-1}}(\Omega)$ be the distributional vector valued divergence defined as follows:
\begin{align*}
\skp{\operatorname{\boldsymbol \nabla}\cdot\, \boldsymbol{\tau},\Bv}:=-\int _\Omega \boldsymbol{\tau} : \CE(\Bv)dx \qquad \forall\Bv \in \BH_0^1(\Omega),
\end{align*}
{where $\boldsymbol{H^{-1}}(\Omega)$ denotes the dual of $\boldsymbol{H_0^{1}}(\Omega)$}. 
{We also define the positive part of a scalar $a \in \R $ as follows
$$
[a]_+:=
\begin{cases}
 a & \text{if}\, \,\,a >0\\
 0&\text{o.w},
\end{cases}
$$
satisfying the following properties \cite[Lem 3.2]{burman2017penalty}
\begin{align}\label{eq.po}
\left([c]_+-[d]_+\right) \left( c-d\right) &\ge \left([c]_+-[d]_+\right)^2\qquad &&\forall c,\,d \in \R,\\
 \label{eq.po2}
\abs{[c]_+-[d]_+} &\le \abs{c-d}\qquad &&\forall c,\,d \in \R.
\end{align}}
{ Finally, throughout this paper, the notation $\lesssim$ indicates $\le$ up to a constant $C>0$.}\\

\subsection{Model Problem}
Let $I:= (0,T]$ be a time interval where $T>0$  is an arbitrary real number. 
Let $\varphi  : \Omega \times I \rightarrow [0,1] $  be the internal variable 
describing  the irreversible damage \mp { (${\varphi=0  }$  indicates completely  damaged material and $\varphi =1$  denotes  the unbroken material)},
$\Bu : \Omega \times I\rightarrow \R ^d $ be the displacement vector, and   $\vartheta : \Omega  \times I \rightarrow \R$ be the absolute temperature function. 
 Then, the model problem  is presented as a  PDE system consisting of three coupled nonlinear equations introduced in \cite{MR3449619, MR3365562} (the reduced version) and \cite{MR3842151} as follows:
	\begin{subequations}
	\begin{alignat}{3}
	\label{The elasticity equation}
	{\partial _{tt} } \Bu		- \operatorname{\boldsymbol \nabla}\cdot\left( (g(\varphi ) +\kappa) \CA \left({ \CE } (\Bu \right) 
	)-\rho \vartheta \mathbf{I} \right) &=\Bf \hspace{1cm}&&\text{in} \,\,\Omega \times I,\\
	\label{	The phase-field equation-a}
	-\ell \Delta \varphi + \frac{1}{\ell} \varphi+ \frac{1}{\mathcal{G}_c}g_c (\varphi,  \CE (\Bu ))&\mp{\ge}0&&\text{in} \,\,\Omega \times I,\\
	\label{The heat equation}
	\partial _t \vartheta + \rho \vartheta \,\nabla \cdot \partial _t \Bu-\nabla \cdot \left(K(\vartheta) \nabla \vartheta\right) &= \gamma &&\text{in} \,\,\Omega \times I,\\
	\label{The complementarity condition}
{	\partial _t \varphi 	\left(	-\ell \Delta \varphi + \frac{1}{\ell} \varphi+ \frac{1}{\mathcal{G}_c}g_c (\varphi,  \CE (\Bu )) \right) }&=0&&\text{in} \,\,\Omega \times I,\\
	\Bu&=0 &&\text{on} \,\,\Gamma  \times I,&\\
	\nabla \varphi \cdot \Bn &=0 &&\text{on} \,\,\Gamma \times I,\\
	\left( K(\vartheta) \,\nabla \vartheta \right) \cdot \Bn  &=\overline{{\gamma}} &&\text{on} \,\,\Gamma \times I,\\
	\label{the heat equation-boundary condition}
	\Bu (\cdot , 0)=\Bu_0, \quad\partial_t \Bu (\cdot , 0)&=\Bv_0 &&\text{in} \,\,\Omega\times \{0\},\\
	\label{The heat equation-initial condition}
	\varphi (\cdot , 0)=\varphi_0, \quad	\vartheta (\cdot , 0)&=\vartheta_0 \quad
	&& \text{in}
	\,\,\Omega\times \{0\},
	\end{alignat}
	\end{subequations}
{where $\Bn$ denotes the outward unit vector normal to
	$\Gamma$.}
	 In  \eqref{The elasticity equation}, the parameter $ \rho > 0 $ is the thermal expansion constant, the function  $g$ is defined as $g(\varphi):= \varphi^2$,
	  $\kappa$ is a positive stability constant for the bulk regularization,  
	  and 
  $\Bf \in \BL^2 (\Omega) $ is the source term. 
In \eqref{	The phase-field equation-a}, $\ell >0$ is the length scale (i.e., damage regularization) parameter, { ${\mathcal{G}_c}$ is a damage dependence positive constant,}
and the function
$$g_c (\varphi,  \CE (\Bu )):=g ^ {\prime} (\varphi)  \,\mathcal{B} \left( \CE (\Bu ) \right)=\frac{1}{2}g ^ {\prime} (\varphi) \CA (\CE (\Bu)): \CE(\Bu), $$ is the  Cauchy stress. 
\mp{Moreover, the damage  indicator $\varphi$ fulfils the irreversibility condition, i.e., $\partial _t \varphi  \le 0$ and  we assume $0\le \varphi _0 \le 1$. Using these assumptions as well as $\varphi \ge 0$, one can easily see that  $0 \le \varphi \le 1$.  }
%
In \eqref{The heat equation}, $K$ is called the  heat conductivity function, and $\gamma \in L^2 (\Omega)$ is the heat source term.
	Following \cite[Eq. 11.a]{MR3666698}, for all $\boldsymbol{ \tau}  \in  [L^2 (\Omega )]^{d\times d}$ there holds
	\begin{align}
	g_c (z_1,  \boldsymbol{ \tau})(z_1-z_2)\ge g (z_1) \CA(\boldsymbol{ \tau}):\boldsymbol{ \tau}-g(z_2)\CA(\boldsymbol{ \tau}):\boldsymbol{ \tau} \qquad \forall z_1,z_2 \in \R,
	\end{align}
	and hence 
	\begin{align}
	g_c (z_1,  \boldsymbol{ \tau})(z_1-z_2)-g_c (z_2,  \boldsymbol{ \tau})(z_1-z_2)\ge 0.
	\end{align} 
	In the following, we mention the required assumptions on the heat conductivity function, the initial conditions, and the source terms.
\begin{assumption}\label{assumption-initial data} In order to proceed further with the main results of this paper, we need to impose these assumptions:
	\begin{itemize}
	\item  We  assume  the heat conductivity function $K :  \R \rightarrow \R$ is Lipschitz continuous and satisfies the following inequalities
	\begin{align*}
	\exists\, \beta \in (1, \beta _d) &\quad \exists \, c_0>0 \quad \text{s.t} \quad \forall \zeta \in \R^d \quad c_0 (\abs{\vartheta} ^\beta +1)\abs{\zeta} ^2 \le K (\vartheta) \zeta \cdot \zeta, \\
	\exists\, \beta \in (1, \beta _d)& \quad \exists \,c_1,\, c_2>0 \quad \text{s.t} \quad \forall \zeta \in \R^d:\\ 
&\quad\quad c_1 (\abs{\vartheta}^ \beta +1) \le  \abs{K (\vartheta) } \le c_2 (\abs{\vartheta}^ \beta +1),
	\end{align*}
	where $\beta_d  =2$ for $d=2$ and $\beta _d =5/3 $ for $d=3$.
	\item
	Moreover, the source and  loading terms satisfy the following assumptions
	\begin{align*}
	&\Bf \in H^1 (I;\BL^2 (\Omega)),\,\,\, \gamma \in H^1 (I, L^2 (\Omega)),\,\,\, \gamma  \ge0,\\ 
 &\overline{\gamma} \in L^1(I;L^2 (\Gamma)), \quad\overline{\gamma} \ge 0 \quad \text{a.e. \,\,\, in}  \, \Gamma \times  I,
	\end{align*}
where $(\BH^1(\Omega))^\prime$ denotes the dual space of $\BH^1(\Omega)$.
	\item 	
We also impose the following assumptions on the initial data
\begin{align*}
{\Bu _0 \in \BH_0^2 (\Omega), \qquad \Bv_0 \in \BH_0^1 (\Omega), \qquad \vartheta _0 \in L^2 (\Omega).}
\end{align*}
\item In general, we are allowed to assume 
$\kappa=\mathcal{O}(\ell)$, $h=\mathcal{O}(\ell)$, and  $\kappa \ll \ell$.
\end{itemize}
\end{assumption}
\section{Variational formulation in space}\label{section3}
In order to present a variational formulation (see e.g., \cite{MR3365562, MR3842151}) for \eqref{The elasticity equation}-\eqref{The heat equation-initial condition} 
with respect to the spatial variable, we introduce the following function spaces
\begin{align*}
	\BV =\BH^1_0(\Omega), \qquad W:=H^1(\Omega)\qquad  {W _+:=\left\lbrace \varphi \in  H^1(\Omega) \quad : \quad  \varphi \ge0 \quad{\text{a.e. \,\, in} \,\,\Omega} \right\rbrace },   \qquad Z :=H^1(\Omega),
\end{align*}
where the space  $\BV$  is equipped with the norm $\norm{\,\cdot\,}_ {\BH^1 (\Omega)}$, and the spaces  $W$ and $Z$ are both equipped with the norm $\norm{\,\cdot\,}_ {H^1 (\Omega)}$. 
Then, the variational formulation for \eqref{The elasticity equation}-\eqref{The heat equation} reads as: For all $t \in I$,  find $(\Bu(\cdot,t), \varphi(\cdot,t), \vartheta(\cdot,t)) 
\in \BX :=\BV \times {W_+} \times Z$  such that
\begin{subequations}
	\begin{alignat}{3}
	\label{variational-elasticity.eq}
	\int_\Omega \partial _{tt} \Bu \,\Bv dx+\int_{\Omega} \left(  g(\varphi  ) + \kappa\right)  \CA \left( \CE (\Bu )\right):  \CE(\Bv) \,dx &-\rho 		\int_{\Omega}  \vartheta\, \BI: \CE(\Bv) \,dx\\
	&\quad \nonumber=	 \int_{\Omega} \Bf (t)\cdot \Bv dx 
	\quad \forall \Bv \in \BV, \\
	\label{variational-phase-field.eq}
{	\ell \int_{\Omega} \nabla \varphi  \cdot \nabla (w-\varphi )\, dx + \frac{1}{\ell} \int_{\Omega}  \varphi  \,  (w-\varphi ) \,dx} &
\mp{	\ge  }\frac{-1}{\mathcal{G}_c} \int_{\Omega}g_c (\varphi,  \CE (\Bu )) \, (w-\varphi ) \,  dx  \quad \forall w \in W_+,\\
	\label{variational-heat.eq}
	\nonumber
	\int_{\Omega}\, \partial _t \vartheta \, z \,dx +\int_{\Omega} K(\vartheta) \nabla \vartheta\cdot \nabla z \,dx&+\rho \int _ \Omega \vartheta \nabla \cdot \partial _t \Bu \,z\, dx\\
	&
	+ \int _ {\Gamma } \overline{\gamma} \, z \,ds= \int_{\Omega} \gamma (t )\, z \,dx \quad \forall z \in Z,
	\end{alignat}
\end{subequations}
plus the initial conditions \eqref{the heat equation-boundary condition} and 
\eqref{The heat equation-initial condition}.
In the next lemma, we mention some results on the  existence and regularity of the solutions of \eqref{variational-elasticity.eq}--\eqref{variational-heat.eq}.
\begin{lemma}(Existence and regularity  of the   solutions of the variational formulations \eqref{variational-elasticity.eq}-\eqref{variational-heat.eq}) (see e.g.,  \cite[Thm.~ 3.7]{MR3449619}). 
	Let $\Omega$ be a Lipschitz continuous domain, and let  all the conditions mentioned in Assumption \ref{assumption-initial data} be satisfied for the right-hand side terms and the boundary conditions. Then, for every vector $(\Bu_0,  \Bv_0, \vartheta_0)$ satisfying Assumption \ref{assumption-initial data}, there exists a solution $(\Bu, \vartheta,  \varphi )$
	such that 
{\begin{align*}
	\Bu & \in H^1(I ; \BH_0^2 (\Omega)) \cap W ^ {1,\infty} (I;\BH_0^1 (\Omega) )\cap H^2 (I; \BL^2 (\Omega)),\\ 
\varphi &\in L^\infty (I; H^1 (\Omega)) \cap H^1 (I ; L^2 (\Omega)),\\
\vartheta & \in L^2 (I; H^1 (\Omega)) \cap L^\infty (I ; L^2 (\Omega)) \cap H^1(I; (H^1(\Omega))^*) ,\quad
\vartheta >0 \quad \text{a.e.\,in } \quad I \times \Omega.
	\end{align*}}
\end{lemma}
Next, we use the properties of the trace-free tensor $\operatorname{\mathbf{dev}}$ to prove the Lipschitz continuity of the operator $\CA$.
First, we mention the following properties of the deviatoric operator \cite{MR2873244, MR3576569}:
\begin{align}
	\label{prop-dev-1}
	\norm{	\operatorname{\mathbf{dev}} \left( \boldsymbol{\tau}\right)  } _{\BL ^2 (\Omega)}&\le \norm{	\boldsymbol{\tau}}_{\BL ^2 (\Omega)} &&\forall \boldsymbol{\tau}\in  [ L^2 (\Omega )]^ {d\times d},\\
	\label{prop-dev-2}
	\skp{\operatorname{\mathbf{dev}} \left( \boldsymbol{\tau}\right)  , \boldsymbol{\sigma}} &= \skp{ \boldsymbol{\tau}  ,\operatorname{\mathbf{dev}} \left( \boldsymbol{\sigma} \right)} && \forall \boldsymbol{\tau}, \,\boldsymbol{\sigma} \in  [ L^2 (\Omega )]^ {d\times d},\\
\nonumber
	\skp{\operatorname{\mathbf{dev}} \left( \boldsymbol{\tau}\right)  , \operatorname{\mathbf{dev}} \left( \boldsymbol{\tau}\right)}&=\skp{\operatorname{\mathbf{dev}} \left( \boldsymbol{\tau}\right)  , \boldsymbol{\tau}}\\
	\label{prop-dev-3}&=\skp{\boldsymbol{\tau}, \boldsymbol{\tau}} -\frac{1}{d}\skp{\operatorname{tr} (\boldsymbol{\tau}),\operatorname{tr} (\boldsymbol{\tau}) }&&\forall \boldsymbol{\tau} \in  [ L^2 (\Omega )]^ {d\times d}.
\end{align}
Here, we prove that the operator $\CA$ is Lipschitz continuous and elliptic.
\begin{lemma}\label{Lem-holder contin}
	The 
	operator $\CA$ defined in \eqref{operator-A} satisfies the following property:
	\begin{align*}
		\norm{\mathcal{A}(\boldsymbol{\tau})-\mathcal{A}(\boldsymbol{\sigma})}_{\BL ^2 (\Omega)} &\le C _ {\mu, \gamma} \norm{\boldsymbol{\tau}-\boldsymbol{\sigma}}_{\BL ^2 (\Omega)}  &&\qquad\forall \boldsymbol{\tau}, \,\boldsymbol{\sigma}  \in  [ L^2 (\Omega )]^ {d\times d},
		\\
	 C_{{ell},\CA}	\norm{\boldsymbol{ \tau}}^2_{\BL^2(\Omega)}&\le  \skp{\CA(\boldsymbol{ \tau}),\boldsymbol{ \tau}}&&\qquad\forall \boldsymbol{\tau}\in  [ L^2 (\Omega )]^ {d\times d},
	\end{align*} 
	where
	$C _ {\mu, \lambda}:={2\lambda d+2 \mu } $, and $C_{{ell},\CA} :=\frac{1}{2\mu}$.
\end{lemma}
\begin{proof}
	From the definition of $ \mathcal{A}$  and the triangle inequality we have
	\begin{align}
		\label{proof-hoelder-1} \norm{\mathcal{A}(\boldsymbol{\tau})-\mathcal{A}(\boldsymbol{\sigma})}_{\BL ^2 (\Omega)}  &\le {\lambda}\norm{\left( \operatorname{tr}(\boldsymbol{\tau})-\operatorname{tr} (\boldsymbol{\sigma})\right):\BI }_{\BL ^2 (\Omega)} + 2\mu \norm{\boldsymbol{\tau}-\boldsymbol{\sigma}}_{\BL ^2 (\Omega)}.
	\end{align}
	It follows by the definition of $\operatorname{\mathbf{dev}}$ that  $\operatorname{tr} (\boldsymbol{\tau}):\BI= d \left(\boldsymbol{\tau}:\BI-\operatorname{\mathbf{dev}}(\boldsymbol{\tau}) \right)$. Using  this and  \eqref{prop-dev-1} for \eqref{proof-hoelder-1}, we get 
	\begin{align}
		\nonumber\norm{\mathcal{A}(\boldsymbol{\tau})-\mathcal{A}(\boldsymbol{\sigma})}_{\BL ^2 (\Omega)}  &\le {\lambda d}\norm{\boldsymbol{\tau}-\boldsymbol{\sigma}}_{\BL ^2 (\Omega)}+ {\lambda d}\norm{\operatorname{\mathbf{dev}}(\boldsymbol{\tau})-\operatorname{\mathbf{dev}}(\boldsymbol{\sigma})}_{\BL ^2 (\Omega)}	\\
		\nonumber
		&+  2\mu  \norm{\boldsymbol{\tau}-\boldsymbol{\sigma } }_ {\BL^2 (\Omega)}
	\le \left( 2 {\lambda d+ 2 \mu }\right) \norm{\boldsymbol{\tau}-\boldsymbol{\sigma}}_{\BL ^2 (\Omega)},
	\end{align}
	denoting $C _ {\mu, \lambda}:={2\lambda d+2\mu } $ completes the first part of proof. 
	The second part can be easily proven considering the definition of $\CA$.
\end{proof}
	{One of the inequalities that we use repeatedly thorough this paper is the following 
		special case of the Gagliardo–Nirenberg inequality (see e.g., \cite[Prop.  III.2.35]{MR3014456} and \cite[P. 125]{MR109940})
		\begin{align}\label{Gaglias}
		\norm{v}_ {L^q(\Omega)} \lesssim \norm{v}^\alpha _{H^1(\Omega)}\norm{v}^ {1-\alpha}_{L^2(\Omega)} \qquad \forall v \in H^1(\Omega),\quad \frac{1}{q}= \frac{1}{2}-\frac{\alpha}{d}\quad \alpha \in [0,1].
		\end{align}
		Moreover, we recall the  continuous embeddings $W ^{1,q} (\Omega) \subset L^\infty (\Omega)$ for $q >d$ and $H^1(\Omega) \subset L^r(\Omega)$ for $1 \le r \le {2d}/{(d-2)}$. (see e.g.,  \cite[Thm. 1.20]{MR3014456}) }

\subsection{Semi-discretization in time}
For the time discretization, we use a  finite difference scheme, and  define the following time points and subintervals with the step size $\tau = t_k-t_{k-1}= \frac{T}{M}$:
$$0 =t_0 <t_1 < \cdots <t_M =T,\qquad I_k = (t_ {k-1}, t_k].$$ 
Then,  for sufficiently smooth function $v  : [0,T] \rightarrow H^ 1 (\Omega)$, we introduce the following notations
\begin{align*}
v^k := v (t_k), \quad 
\partial _\tau ^k v &:= \frac{v ^{k+1}-v ^k}{\tau}, 
\quad 
\partial _ {\tau \tau}  ^k v := \frac{v ^{k-1} -2 v ^ {k} + v^{k+1}}{\tau ^2}=\frac{ \partial _\tau ^k v-\partial _\tau ^{k-1} v}{\tau}, \\ 
 \delta _\tau ^k  v &:= \frac{v ^ {k+1}- v ^ {k-1}}{2 \tau}=\frac{ \partial _\tau ^k v+\partial _\tau ^{k-1} v}{2}.
\end{align*} 
We  set $\Bu ^ {-1}:= \Bu_0- \tau \Bv_0$ and  $\vartheta ^0 := \vartheta _0$.
Then, the semi-discretized weak formulation  for  \eqref{variational-elasticity.eq}-\eqref{variational-heat.eq} reads as: For $k \in \{ 0,1,\cdots, M-1\}$, find $(\Bu ^{k+1}  , \varphi ^{k+1}, \vartheta  ^{k+1}) \in \BV \times W \times Z$ such that
	\begin{subequations}
	\begin{alignat}{8}
\label{variational-elasticity-semi discrete.eq}
\nonumber
\int_\Omega \partial ^k _{\tau \tau} \Bu  \,\Bv  dx&+\int_{\Omega} \left(   g(\varphi ^ {k+1} )+\kappa\right)  \CA \left( \CE (\Bu ^ {k+1})\right): \CE(\Bv) \,dx-\rho		\int_{\Omega} \vartheta ^k \,\BI: \CE(\Bv)\,dx\\
&=	 \int_{\Omega} \Bf ^k\cdot \Bv dx 
&&\hspace{-1cm} \forall \Bv \in \BV,\\ \nonumber
\ell \int_{\Omega} \nabla \varphi ^{k+1}   \cdot \nabla w \, dx &+ \frac{1}{\ell} \int_{\Omega}  \varphi^{k+1}  \,  w \,dx \mp{+{\gamma_0}\int_{\Omega}  [\varphi^{k+1}- \varphi ^{k}]_+\, \,w\,dx }\\&
\label{variational-phase-field-semi discrete.eq}
 = \frac{-1}{\mathcal{G}_c} \int_{\Omega}g_c (\varphi ^k, { \CE } (\Bu ^{k} )) \, w \,  dx &&\hspace{-1cm}\forall w \in W,\\
\label{variational-heat-semi discrete.eq}
\nonumber
  \int_{\Omega} \partial ^k _\tau \vartheta \, z \,dx &+ \int_{\Omega} K(\vartheta ^ {k+1})\,\nabla \vartheta ^ {k+1} \cdot \nabla z \,dx+ \rho  \int_ \Omega\vartheta ^ k\nabla \cdot \delta_\tau ^ k \Bu\,z \,dx\\&
\quad+ \int _ {\Gamma } \overline{\gamma} ^ {k}\, z \,ds= \int_{\Omega} \gamma  ^ {k}\, z \,dx &&\hspace{-1cm} \forall z \in Z,
	\end{alignat}
\end{subequations}
\mp{where $\gamma_0  >0 $ is a sufficiently large penalization parameter.}
%
\subsubsection{A priori estimates}
Among  the essential tools in the proof of the  main result of this paper are  a priori estimates for the strain tensor of the semi-discrete solution of \eqref{variational-elasticity-semi discrete.eq}, the   semi discrete solution of \eqref{variational-heat-semi discrete.eq} as well as  its gradient at each time step, i.e., $\CE (\Bu ^k)$, $\vartheta ^k$ and $\nabla \vartheta ^k$ $(\forall k =1,2, \cdots, M)$, respectively.
The above  discussion leads to the following lemma.
\begin{lemma}(A priori estimates for the heat function, the displacement vector field, and the strain tensor of displacement). 
	\label{Lem-a- priori estimate-semi discrete}
	Let $ L \in \{ 1,2,\cdots, M\}$ and $\Bu ^L \in  \BH_0 ^ {1} (\Omega)$, and $\varphi ^{L}, \vartheta ^{L} \in H ^1 (\Omega)$ be the solutions of \eqref{variational-elasticity-semi discrete.eq}-- \eqref{variational-heat-semi discrete.eq} in the $L$-th time step.
%
	 Then, for  $\Bf ^L \in \BL^2 (\Omega), \, \overline{\gamma} ^L \in L^2(\Gamma)$ and $\gamma ^L \in L^2 (\Omega)$ we have the following a priori estimates for the heat function, the displacement vector field, and the strain tensor, respectively
\begin{align}
\label{eq-a- priori estimate-semi discrete-heat}
\norm{\vartheta ^ {L}}^2_ {L^2 (\Omega)} + \tau \norm { \nabla \vartheta ^ {L}} ^2_ {\BL^2 (\Omega)}&\le C_ {st, \vartheta} \tau \mathcal{L}_{1,L},\\
\label{eq-a- priori estimate-semi discrete-disp}
\norm{\Bu ^ L}_{\BL ^2 (\Omega)}&\le C_ {st, \Bu}\left(  \tau^2 \mathcal{L}_{1,L}+ \tau \norm{ \Bv_0}_{\BL ^2 (\Omega)}^2\right) ,
\end{align}	
where 
\begin{align*}
\mathcal{L}_{1,L} :=\sum_{k=0}^{L-1} \norm{\overline{\gamma }^k}^2_ {L^2 (\Gamma)}
+\gamma_\kappa  \sum_{k=0}^{L-1}  \norm{\Bf ^ k }^{(\beta +2)^2 /(\beta +1)^2}_{\BL^2 (\Omega)} 
+\tau   \sum_{k=0}^{L-1} \norm{\gamma ^k}^2_{L^2 (\Omega)}
+((\gamma _ \kappa +1)\abs{\Omega}^2),  
\end{align*}    
and  the nonegative constants $C_ {st, \vartheta}$ and $C_ {st, \Bu}$ are independent of $\tau$, $h$, $\ell$, and $\kappa$, and $\gamma_ \kappa$ is defined as 
$\gamma_ \kappa := \frac{\rho ^4   }{\kappa }$.
	Furthermore, we have the following  $\BL^2$-norm estimates for $\CE (\delta _\tau ^{L-1}\Bu )$  and $\CE (\Bu ^L) $
	\begin{align}
	\nonumber
 \frac{\tau ^2}{2}\norm{\CE (\delta _\tau ^ {L-1}\Bu )) }^2_ {\BL ^2 (\Omega)}+\norm{\CE (\Bu ^ {L})) }^2_ {\BL ^2 (\Omega)} &\le C _ {st, \CE\Bu} \left( \frac{\rho ^2\tau^2 }{ \alpha _ \kappa } \mathcal{L}_{1,L}+ \frac{\tau ^2}{2\alpha _{\kappa}}
 \sum_{k=0}^{L-1}\norm{\Bf^k}_{\BL ^2 (\Omega)}^2\right) ,
	\end{align} 
	where $\alpha _ \kappa :={\kappa\, C_{\text{ell}, \CA}}$, and $C_{\text{ell}, \CA}$ is the ellipticity constants of the linear operator  $\CA$ dependent on the Lam\'e parameters and   the nonegative  constant $C_ {st, \CE \Bu}$ is independent of $\tau$, $h$, $\ell$, and $\kappa$.
\end{lemma}
\begin{proof}
Choosing  $\Bv = \delta _ \tau ^k \Bu, \; k=1,2,\cdots,M-1$ as the test function in \eqref{variational-elasticity-semi discrete.eq}  and since the operator $\CA$ is symmetric,
it follows that
 \begin{align}
 \nonumber 
 \frac{1}{2 \tau} & \left(\norm {\partial _ \tau ^ {k} \Bu}^2_ {\BL^2 (\Omega)}\right. -\left. \norm {\partial _ \tau ^ {k-1} \Bu}^2_ {\BL^2 (\Omega)}\right)+\frac{\kappa \tau }{4}\skp{  \CA \left( \CE (\delta _\tau ^k\Bu )\right), \CE(\delta _\tau ^k\Bu)}\\& 
 \nonumber\quad+\frac{\kappa  }{2\tau}\skp{  \CA \left( \CE (\Bu ^ {k+1} )\right), \CE(  \Bu^ {k+1})}\\& 
\nonumber
 \le \frac{\kappa \tau }{2}\abs { \skp{  \CE(  \Bu^ {k+1}), \CA \left( \CE (\Bu ^ {k-1} )\right) }}+\frac{\kappa}{2}\abs {\skp{\CA \left( \CE (\Bu ^ {k-1} )\right), \CE (\delta ^k _\tau \Bu)}}
 \\& \quad+  \rho  \int_{\Omega}  \vartheta ^k   \nabla \cdot \delta _\tau ^k \Bu\,dx 
 +  \int_ \Omega \Bf ^k \delta ^k _ \tau \Bu\, dx.
 \end{align}
Integration by parts gives us
 $$\int_{\Omega} \nabla  \vartheta ^k \cdot  \delta _ \tau ^k \Bu\,dx =\int_{\Omega} \vartheta ^k \,\nabla \cdot \delta^k_\tau \Bu\,dx, $$ 
which based on that, and  applying
 the ellipticity property of the operator $\CA$ from Lemma \ref{Lem-holder contin}, and Young's inequality lead to the following estimate
\begin{align}
\label{k+1 step for u-1}
\nonumber \frac{1}{2 \tau} & \left(\norm {\partial _ \tau ^ {k} \Bu}^2_ {\BL^2 (\Omega)}\right. -\left. \norm {\partial _ \tau ^ {k-1} \Bu}^2_ {\BL^2 (\Omega)}\right)+ \frac{\tau\kappa C_{ell, \CA}}{4}\norm{\CE (\delta _\tau ^ k\Bu )) }^2_ {\BL ^2 (\Omega)}\\&
\nonumber\quad + \frac{\kappa C_{{ell},\CA}}{2 \tau }\norm{\CE (\Bu ^ {k+1})) }^2_ {\BL ^2 (\Omega)}\\&
\nonumber 
 \le { \rho ^2}\tau \norm{\nabla \vartheta ^ k }^2_{\BL^2 (\Omega)}+ \frac{1}{4 \tau } \left( \norm{\partial _ \tau ^ {k} \Bu }^2_ {\BL ^2 (\Omega)}+\norm{\partial _ \tau ^ {k-1} \Bu  }^2_ {\BL ^2 (\Omega)}\right)   \\&
 \nonumber \quad+\frac{1}{8 \tau } \left(\norm {\partial _ \tau ^ {k} \Bu}^2_ {\BL^2 (\Omega)}\right. +\left. \norm {\partial _ \tau ^ {k-1} \Bu}^2_ {\BL^2 (\Omega)}\right)\\&
 \nonumber
 \quad +\frac{\tau }{2}\norm{\Bf}_{\BL ^2 (\Omega)}^2+\frac{\kappa \tau  C_{\mu,\lambda} }{2  }\norm{\CE (\Bu ^{k-1})) }^2_ {\BL ^2 (\Omega)}+\frac{\kappa \tau   }{2  }\norm{\CE (\Bu ^{k+1})) }^2_ {\BL ^2 (\Omega)}\\& 
 \quad+
 \frac{\kappa  C_{\mu,\lambda} }{  \tau C_{ell, \CA} }\norm{\CE (\Bu ^{k-1})) }^2_ {\BL ^2 (\Omega)}
+\frac{\kappa \tau C_{ell, \CA}}{8} \norm{\CE(\delta _ \tau ^k \Bu)}^2_ {\BL ^2 (\Omega)}.
 \end{align}
Then, for sufficiently small $\tau$ it is immediate that 
\begin{align}
\label{Proof- appriori-46}
\nonumber
\frac{\tau\alpha _\kappa}{4}\norm{\CE (\delta _\tau ^ k\Bu )) }^2_ {\BL ^2 (\Omega)}+ \frac{\alpha_\kappa}{2 \tau } \norm{\CE (\Bu ^ {k+1})) }^2_ {\BL ^2 (\Omega)}&\le {\rho ^2 \tau } \norm{ \nabla \vartheta ^ k }_{\BL^2 (\Omega)}^2+ \frac{\tau }{2}\norm{\Bf^k}_{\BL ^2 (\Omega)}^2\\&
\quad+ \beta_ {\kappa , \tau} \norm{\CE (\Bu ^{k-1})) }^2_ {\BL ^2 (\Omega)},
\end{align}
where $\alpha _ \kappa :=\kappa   C_{ell, \CA} /2$ and $\beta_ {\kappa , \tau}:= {\tau ^ {-1}\left( \kappa C_{ell, \CA}^ {-1} C_{\mu,\lambda}+\kappa\tau^2 C_{\mu,\lambda}/{2   }\right)  }$. Finally, the discrete Gronwall inequality results in the following estimate
\begin{align}
\label{Proof- appriori-46-a}
\frac{\tau ^2}{2}\norm{\CE (\delta _\tau ^ k\Bu )) }^2_ {\BL ^2 (\Omega)}+\norm{\CE (\Bu ^ {k+1})) }^2_ {\BL ^2 (\Omega)}&\le C_{G, \Bu}\frac{2 \rho ^2\tau^2 }{ \alpha _ \kappa}\norm{\nabla \vartheta ^ k }_{L^2 (\Omega)}^2+ \frac{\tau ^2 }{2 \alpha _ \kappa}\norm{\Bf^k}_{\BL ^2 (\Omega)}^2,
\end{align}
where $C_{G, \Bu}$ is a positive  constant independent of $\tau$, $h$, $\ell$, and $\kappa$.
On the other hand, using $z = \vartheta ^ {k+1} $ as the test function in \eqref{variational-heat-semi discrete.eq}
 and via the H\"older inequality \cite[Prop. II.2.18]{MR2986590},    we deduce that
\begin{align}
\label{proof-apriori-47}
\nonumber \frac{1 }{\tau}&\left( \norm{ \vartheta ^ {k+1}-\vartheta ^ k} ^2 _ {L^2 (\Omega)}\right.  +\left.  \norm{ \vartheta ^ {k+1}} ^2 _ {L^2 (\Omega)}-\norm{ \vartheta ^ {k}} ^2 _ {L^2 (\Omega)}\right) \\
\nonumber 
&+ c_1  \norm {\abs{\vartheta ^ {k+1}}^ {\beta/2}  \nabla \vartheta ^ {k+1}} ^2 _ {\BL^2 (\Omega)}+c_1  \norm { \nabla \vartheta ^ {k+1}} ^2 _ {\BL^2 (\Omega)}
\\
\nonumber 
& \le \frac{\rho C_K}{2}  \norm{ \vartheta ^ {k+1}}  _ {L^3 (\Omega)}\norm{ \vartheta ^ {k}}  _ {L^6 (\Omega)} \norm{\CE (\delta _\tau ^{k}\Bu)  }_ {\BL ^2 (\Omega)} +\tau   \norm{\gamma ^k}^2_{L^2 (\Omega)} \\
&
\quad
+ \frac{1}{4 \tau} \norm{ \vartheta ^ {k+1}} ^2 _ {L^2 (\Omega)}
+\frac{1}{c_1}\norm{\overline{\gamma }^k}^2_ {L^2 (\Gamma)}+\frac{c_1}{4} \norm{ \nabla  \vartheta ^ {k+1}} ^2 _ {\BL^2 (\Omega)},
\end{align}
where $C_K$ is the constant of Korn's inequality.
For each $k \in \{ 0,1,\cdots,M-1\}$, we define the auxiliary variable  $\Phi ^ {k+1 }:=   \vartheta^ {k+1}+1 $ and accordingly we denote $w^ {k+1}:=(\Phi^ {k+1}) ^ {(\beta +2)/2 }$, which results in 
\begin{align}
\nonumber
\frac{c_1}{2}\norm{(\abs{\vartheta ^ {k+1}}^ {\beta /2} +1) \nabla \vartheta ^ {k+1}}^ 2 _{\BL^2 (\Omega)} &\ge \frac{c_1}{2} \norm{ (\vartheta ^ {k+1}+1)^{\beta  /2}\nabla  \vartheta ^ {k+1}}^ 2 _{\BL^2 (\Omega)}\\&\ge\frac{1}{\beta+2} \norm{ \nabla w ^ {k+1} }^ 2 _{\BL^2 (\Omega)}.
\end{align}
The first inequality is due to the Bernoulli inequality and 
 the last inequality is a direct conclusion of 
$\nabla w ^ {k}=(\beta+2)/2\,(\Phi ^k) ^ {\beta/2}\nabla \vartheta^ {k }$.
Applying the above estimate allows us to obtain  the following inequality  for the last two terms in the left hand side of \eqref{proof-apriori-47}
\begin{align}
\nonumber
 c_1  \norm {\abs{\vartheta ^ {k+1}}^ {\beta/2}  \nabla \vartheta ^ {k+1}} ^2 _ {\BL^2 (\Omega)}&+c_1  \norm { \nabla \vartheta ^ {k+1}} ^2 _ {\BL^2 (\Omega)}\\
 \nonumber
  &\ge \frac{c_1}{2} \left(  \norm{\abs{\vartheta ^ {k+1}}^ {\beta /2} \nabla \vartheta ^ {k+1}}^ 2 _{\BL^2 (\Omega)}+\norm{\nabla \vartheta ^ {k+1}}^ 2 _{\BL^2 (\Omega)}\right)\\
 &  \quad+\frac{c_1}{2(\beta+2)} \norm{ \nabla w ^ {k+1} }^ 2 _{\BL^2 (\Omega)}.
\label{proof-apriori-50a}
 \end{align}
On the other hand, applying  the definition of $L^6$-norm, we infer that 
\begin{align}
\label{proof-apriori-50}
\norm{ \vartheta ^ {k+1}}  _ {L^6 (\Omega)}& \le \left(\int _\Omega( w ^ {k+1} ) ^ {\frac{12}{\beta +2}}\right) ^ {1/6}\le \norm{ w^ {k+1}}^{2 /(\beta +2)}  _ {L^3 (\Omega)}.
\end{align}
As a direct result of  \eqref{Proof- appriori-46-a}, we have
   \begin{align}
   \nonumber
   \rho \norm{ \vartheta ^ {k+1}}  _ {L^6 (\Omega)}\norm{ \vartheta ^ {k}}  _ {L^3 (\Omega)}\norm{\CE (\delta _\tau ^{k}\Bu )) }_ {\BL ^2 (\Omega)}  &
   \le
   \nonumber
   \sqrt{2}\rho C^{1/2}_{G, \Bu}\norm{ \vartheta ^ {k+1}}  _ {L^6 (\Omega)}\norm{ \vartheta ^ {k}}  _ {L^3 (\Omega)} \\
&\left( \frac{\sqrt{2}\rho  }{ \sqrt{\alpha _ \kappa }  } \norm{ \nabla \vartheta ^ k }_{\BL^2 (\Omega)}
\quad+ \frac{1  }{\sqrt{ 2 \alpha _\kappa} }\norm{\Bf^k}_{\BL ^2 (\Omega)}
\right),
   \end{align}
   where $C_{G, \Bu}$ is  defined in \eqref{Proof- appriori-46-a}.
   Combining \eqref{proof-apriori-50},  the Bernoulli inequality,
   and the  embedding $L^6 (\Omega) \subset L^ 3(\Omega)$
   we find that
\begin{align}
\nonumber
\rho \norm{ \vartheta ^ {k+1}}  _ {L^6 (\Omega)}&\norm{ \vartheta ^ {k}}  _ {L^3 (\Omega)}\norm{\CE (\delta _\tau ^{k}\Bu )) }_ {\BL ^2 (\Omega)} \\
&
\nonumber\le 
\frac{{2} C ^{1/2}_ {G, \Bu}\rho ^{2}  }{  \sqrt {\alpha _ \kappa}  } \norm{ \nabla \vartheta ^ k }_{\BL^2 (\Omega)} \norm{ w^ {k+1}}^{2 /(\beta +2)}  _ {L^3 (\Omega)}\norm{ w^ {k}}^{2 /(\beta +2)}  _ {L^3 (\Omega)}\\
\nonumber
&\quad+\frac{\rho C ^{1/2}_ {G, \Bu} }{\sqrt{ \alpha _ \kappa }}  \norm{ w^ {k+1}}^{2 /(\beta +2)}  _ {L^3 (\Omega)}\norm{ w^ {k}}^{2 /(\beta +2)}  _ {L^3 (\Omega)} \norm{\Bf ^k}_{\BL ^2 (\Omega)}
 \\& 
\label{proof-apriori-52}
 :=I_1+I_2.
\end{align}
In order to kick back the term $\norm{ w^ {k+1} }_{L^3 (\Omega)} ^ {2 /(\beta +2)} $ into the left hand side of \eqref{proof-apriori-47}, we use Young's inequality twice with the exponents $p={{\beta +2}}$ and $q=\frac{\beta +2}{\beta +1}$   and $p={{\beta +1}}$ and $q=\frac{\beta }{\beta +1}$ and note  the  interpolation inequality from \cite[Lemma II.2.33]{MR2986590} 
to get 
\begin{align}
\nonumber
I_1 &\le \frac{1 }{2(\beta +2)}  \norm{ w^ {k+1}}^{2 }  _ {L^3 (\Omega)}+\gamma _\kappa \norm{ w^ {k}}^{2 /(\beta +1)} _ {L^3 (\Omega)}\norm{ \nabla \vartheta ^ k }^{(\beta +2) /(\beta +1)}_{L^2 (\Omega)}\\
\nonumber
& \le \frac{1 }{2\beta +2}  \norm{ w^ {k+1}}^{2 }  _ {L^3 (\Omega)}+\frac{\gamma _\kappa }{\beta +1}\norm{ w^ {k}}^{2 } _ {L^3 (\Omega)}+\frac{  {\gamma _\kappa}(\beta +1)}{(\beta +2)}\norm{\nabla \vartheta ^ k }^{(\beta +2)\beta /(\beta +1)^2}_{L^2 (\Omega)}\\
&\nonumber \le \frac{1 }{\beta +2}  \norm{ w^ {k+1}}  _ {L^6 (\Omega)}\norm{ w^ {k+1}}  _ {L^2 (\Omega)}+\frac{ \gamma _\kappa}{(\beta +1)} \norm{ w^ {k}}  _ {L^6 (\Omega)}\norm{ w^ {k}}  _ {L^2 (\Omega)}\\
\label{proof-apriori-53}
&\quad+\frac{ \gamma _\kappa(\beta +1)}{(\beta +2)}\norm{ \nabla \vartheta ^ k }^{(\beta +2)\beta /(\beta +1)^2}_{\BL^2 (\Omega)},
\end{align}
where $\gamma_ \kappa :=\frac{4 C _ {G, \Bu}\rho ^4  (\beta +1) }{ (\beta +2)\alpha _ \kappa }$. Applying the Sobolev embedding $H^1 (\Omega ) \subset L^6 (\Omega) $, and Young's inequality give us
\begin{align}
\nonumber
I_1 &  \le \frac{c _1 }{8(\beta +2)}   \norm{ \nabla w^ {k+1}} ^2 _ {\BL^2 (\Omega)}+\frac{c ^2 }{c_1(\beta +2)} \norm{ w^ {k+1}}^2  _ {L^2 (\Omega)} 
\\
&
\quad\nonumber+\frac{c \gamma _\kappa }{2(\beta +2)} \left(  \norm{ \nabla w^ {k}} ^2 _ {\BL^2 (\Omega)}+\norm{ w^ {k}}^2  _ {L^2 (\Omega)}\right)\\   
\nonumber
&\quad+\frac{ \gamma _\kappa(\beta +1)}{(\beta +2)}\norm{\nabla \vartheta ^ k + \bold{1}}^{(\beta +2)\beta /(\beta +1)^2}_{\BL^2 (\Omega)}
\\&\quad
\nonumber
\le \frac{c _1 }{8(\beta +2)}   \norm{ \nabla w^ {k+1}} ^2 _ {\BL^2 (\Omega)}+\frac{c ^2 }{c_1(\beta +2)} \norm{ w^ {k+1}}^2  _ {L^2 (\Omega)} \\
&
\quad
\nonumber+\frac{c \gamma _\kappa }{2(\beta +2)} \left(  \norm{ \nabla w^ {k}} ^2 _ {\BL^2 (\Omega)}+\norm{ w^ {k}}^2  _ {L^2 (\Omega)}\right)\\   
\label{proof-apriori-51}
&\quad+\frac{ \gamma _\kappa(\beta +1)}{(\beta +2)}\norm{\nabla \vartheta ^ k }^2+\frac{c \abs{\Omega}^2\gamma _\kappa }{2(\beta +2)},
\end{align}
where $c$ is a constant independent of $h,\, \tau,\, \ell$, and $\kappa$.
In a procedure similar to $I_1$, we have the following upper bound for $I_2$
\begin{align}
\nonumber
I_2 &\le  \frac{c  _1}{8(\beta +2)}  \norm{ \nabla w^ {k+1}} ^2 _ {\BL^2 (\Omega)}+ \frac{c  ^2}{c_1(\beta +2)} \norm{ w^ {k+1}}^2  _ {L^2 (\Omega)}
\\& \nonumber
+\frac{c \gamma _\kappa }{2(\beta +2)} \left(  \norm{ \nabla w^ {k}} ^2 _ {\BL^2 (\Omega)}+\norm{ w^ {k}}^2  _ {L^2 (\Omega)}\right)
\\
&\quad
\label{proof-apriori-54}
+\frac{ \gamma _\kappa(\beta +1)}{(\beta +2)}\norm{\Bf ^ k }^{(\beta +2)^2 /(\beta +1)^2}_{\BL^2 (\Omega)}.
\end{align}
As a result of Clarkson’s inequality  \cite[Lem. II.2.31]{MR2986590}, the Bernoulli inequality and  again  Clarkson’s inequality, it follows that
\begin{align} 
	\nonumber
	\norm{w^ {k+1} }^{2}_{L^2 (\Omega)} &\le  \norm{ \vartheta^ {k+1}(\vartheta^ {k+1}+1)^{\beta/2}+(\vartheta^ {k+1}+1)^{\beta/2} }^2_ {L^2 (\Omega)}\\& 
	\nonumber
	\le \norm{ \vartheta^ {k+1}(\vartheta^ {k+1}+1)^{\beta/2} }^2_ {L^2 (\Omega)}+\norm{ (\vartheta^ {k+1}+1)^{\beta/2} }^2_ {L^2 (\Omega)}\\& 
	\nonumber
	\label{proof-apriori-55}
	\le \norm{ \vartheta^ {k+1}(\vartheta^ {k+1})^{\beta/2}+\vartheta^ {k+1}  }^2_ {L^2 (\Omega)}+\norm{\vartheta^ {k+1}+1  }^\beta_ {L^2 (\Omega)}
	\nonumber
		\\&\le \norm{(\vartheta^ {k+1})^{1+\beta/2} }^2_ {L^2 (\Omega)}+\norm{ \vartheta^ {k+1}  }^2_ {L^2 (\Omega)}+\norm{\vartheta^ {k+1}+1  }^2_ {L^2 (\Omega)}
		\nonumber
	\\&\le \norm{(\vartheta^ {k+1})^{1+\beta/2} }^2_ {L^2 (\Omega)}+2\norm{ \vartheta^ {k+1}  }^2_ {L^2 (\Omega)}+\abs{\Omega}^2.
	\end{align}
On the other hand, since $\beta\in (1,2)$   with the aid of \eqref{Gaglias} and   Clarkson’s inequality  \cite[Lem. II.2.31]{MR2986590}, we observe that
\begin{align}
\nonumber
\norm{(\vartheta^ {k+1})^{1+\beta/2} }^2_ {L^2 (\Omega)}& \le \norm{\vartheta^ {k+1} }^{\frac{2}{\beta+2}}_ {L^{\beta+1} (\Omega)}\le \norm{\vartheta^ {k+1}+1 }^2_ {L^6 (\Omega)} \\
\nonumber
&\le C\left(  \norm{\nabla \vartheta ^ {k+1}}_ {\BL^2 (\Omega)}^{1/2} \norm{ \vartheta ^ {k+1}+1}_ {L^2 (\Omega)}^{1/2}+\norm{ \vartheta ^ {k+1}+1}_ {L^2 (\Omega)}\right)^2\\
&	\label{proof-apriori-55a} 
\le \frac{c_1}{8} \norm{\nabla \vartheta ^ {k+1}}^2_ {\BL^2 (\Omega)}+\left(8 c_1C^2+2C \right) \left(  \norm{ \vartheta ^ {k+1}}^2_ {L^2 (\Omega)}+\abs{\Omega}^2\right), 
\end{align}
From the definition of $w^k$, we have
\begin{align}
 \nonumber \norm{ \vartheta ^ {k+1}-\vartheta ^ k} ^2 _ {L^2 (\Omega)}&=\norm{ (w ^ {k+1})^ {2/(\beta+2)}-(w ^ {k})^ {2/(\beta+2)}} ^2 _ {L^2 (\Omega)}\\&
 \label{proof-apriori-56}
\ge \norm{ (w ^ {k+1})^ {2/(\beta+2)}} ^2 _ {L^2 (\Omega)}-\norm{ (w ^ {k})^ {2/(\beta+2)}} ^2 _ {L^2 (\Omega)}.
\end{align}
Combining  \eqref{proof-apriori-47}, \eqref{proof-apriori-50a}, \eqref{proof-apriori-52}, 	  \eqref{proof-apriori-51}, \eqref{proof-apriori-54}, \eqref{proof-apriori-55}, \eqref{proof-apriori-55a}, and \eqref{proof-apriori-56} and multiplying both sides of the new estimate in $\tau$, we obtain for sufficiently small $\tau $  that
\begin{align}
\nonumber
\norm{ \right. &\left. \vartheta ^ {k+1}} ^2 _ {L^2 (\Omega)}-\norm{ \vartheta ^ {k}} ^2 _ {L^2 (\Omega)}+\frac{c_1 \tau}{2} \left(\norm{\nabla \vartheta ^ {k+1}}^ 2 _{\BL^2 (\Omega)}+\norm{\abs{\vartheta ^ {k+1}}^ {\beta/2}\,\nabla \vartheta ^ {k+1}}^ 2 _{\BL^2 (\Omega)}\right) \\
\nonumber
& \quad+\frac{c_1\tau}{2(\beta+2)} \norm{ \nabla w ^ {k+1} }^ 2 _{\BL^2 (\Omega)} +\tau \norm{ (w ^ {k+1})^ {2/(\beta+2)}} ^2 _ {L^2 (\Omega)}-\tau \norm{ (w ^ {k})^ {2/(\beta+2)}} ^2 _ {L^2 (\Omega)}\\
\nonumber
& \lesssim \gamma_ \kappa \tau \left( \norm{ \nabla w ^ {k} }^ 2 _{\BL^2 (\Omega)}+\norm{\vartheta^k}^2_ {L^2 (\Omega)}+\norm{ \nabla \vartheta^ {k}   }^2_ {\BL^2 (\Omega)}\right)+\tau\norm{\vartheta^ {k}   }^2_ {L^2 (\Omega)}+\tau\norm{\nabla \vartheta^ {k}   }^2_ {L^2 (\Omega)} \\
\nonumber
&\quad+\gamma _\kappa \tau \norm{\Bf ^ k }^{(\beta +2)^2 /(\beta +1)^2}_{\BL^2 (\Omega)}+\tau ^2  \norm{\gamma ^k}^2_{L^2 (\Omega)} 
+\tau \norm{\overline{\gamma }^k}^2_ {L^2 (\Gamma)}+\tau(\gamma _ \kappa+1)\abs{\Omega}^2.
\end{align}
Summing over  $k$ (with $k=1,2,\cdots,L-1$), 
and using the discrete Gronwall's lemma leads to the inequality 
\begin{align}
\nonumber
\norm{\vartheta ^ {L}}^2_ {L^2 (\Omega)}&+ \tau   \norm { \nabla \vartheta ^ {L}} ^2 _ {\BL^2 (\Omega)}+\frac{\tau}{\beta+2} \norm{ \nabla w ^ {L} }^ 2 _{\BL^2 (\Omega)} \lesssim \tau \sum_{k=1}^{L-1} \norm{\overline{\gamma }^k}^2_ {L^2 (\Gamma)} \\
\label{proof-apriori-57}
&+\gamma _\kappa \tau \sum_{k=1}^{L-1} \norm{\Bf ^ k }^{(\beta +2)^2 /(\beta +1)^2}_{\BL^2 (\Omega)}+\tau ^2  \sum_{k=1}^{L-1} \norm{\gamma ^k}^2_{L^2 (\Omega)}+\tau(\gamma _ \kappa+1)\abs{\Omega}^2. 
\end{align}
Combining this estimate with \eqref{Proof- appriori-46-a} results in the following estimate
\begin{align}
\nonumber 
\frac{\tau ^2}{2}\norm{\CE (\delta _\tau ^ {L-1}\Bu )) }^2_ {\BL ^2 (\Omega)}&+\norm{\CE (\Bu ^ {L})) }^2_ {\BL ^2 (\Omega)} \\
\nonumber&\lesssim \frac{C _ {G, \Bu}\rho ^2\tau^2 }{ \alpha _ \kappa }\left(\sum_{k=1}^{L-1} \norm{\overline{\gamma }^k}^2_ {L^2 (\Gamma)}+\gamma _\kappa  \sum_{k=1}^{L-1}  \norm{\Bf ^ k }^{(\beta +2)^2 /(\beta +1)^2}_{\BL^2 (\Omega)}\right. \\&
\quad \left. +\tau   \sum_{k=1}^{L-1} \norm{\gamma ^k}^2_{L^2 (\Omega)} + \tau(\gamma _ \kappa+1)\abs{\Omega}^2 \right) +  \frac{\tau ^2 }{2 \alpha _ \kappa}
\sum_{k=1}^{L-1}\norm{\Bf^k}_{\BL ^2 (\Omega)}^2.
\end{align}
Finally, the combination of \eqref{k+1 step for u-1}, \eqref{Proof- appriori-46}, and the above estimate as well as a Gronwall's lemma completes the proof.
\end{proof}
\begin{remark}
In case $\kappa =\mathcal O (h)$,  the results of the previous lemma stay valid if and only if ${\tau }{h ^ {-1}}= \mathcal{O}(1)$, which leads to the conditional stability.  
\end{remark}
In the following, we present   {an a priori estimate  for $\CE(\Bu ^ L) $ in the $\CA$-norm.}

\begin{lemma}(An a priori estimate for the strain tensor of displacement 
at the $L$-th time step in $\CA$-norm)
		\label{Lem-a- priori estimate in A-norm}
	Let $ L \in \{ 1,2,\cdots, M\}$, $\Bu  ^L \in  \BH _0^ 1 (\Omega)$, and $\varphi ^L,\, \vartheta ^L \in H ^1 (\Omega)$ be the solutions of \eqref{variational-elasticity.eq},  \eqref{variational-phase-field.eq}, and \eqref{variational-heat.eq}, respectively. Then, for  $\Bf ^L\in \BL^2(\Omega),\, \overline{\gamma}^L \in L^2(\Gamma)$ and $\gamma ^L \in H^1(\Omega)$ we have the following a priori estimate in the $\CA$-norm 
 	\begin{align}
 \nonumber
\norm{\CE (\Bu ^ {L})) }^2_ {\CA} &\le C _ {st, \Bu, \CA} \left( \frac{\rho ^2\tau^2 }{ \kappa } \widehat {\mathcal{L}}_{1,L}+ \frac{\tau ^2}{{\kappa}}
 \sum_{k=0}^{L-1}\norm{\Bf^k}_{\BL ^2 (\Omega)}^2\right) ,
 \end{align}
 where
 \begin{align*}
\widehat{ \mathcal{L}}_{1,L} :=\sum_{k=0}^{L-1} \norm{\overline{\gamma}^k}^2_ {L^2 (\Gamma)}+ \kappa ^{-1} \sum_{k=0}^{L-1}  \norm{\Bf ^ k }^{(\beta +2)^2 /(\beta +1)^2}_{\BL^2 (\Omega)} +\tau   \sum_{k=0}^{L-1} \norm{\gamma ^k}^2_{L^2 (\Omega)}+\abs{\Omega}^2,  
 \end{align*}    
 and  the nonegative constant $C_ {st, \vartheta, \CA}$ is independent of $\tau$, $h$, $\ell$, and $\kappa$.
\end{lemma}
 \begin{proof}
 Choosing  $\Bv = \delta _ \tau ^k \Bu $
 as the test function in \eqref{variational-elasticity-semi discrete.eq},  and in a procedure similar to Lemma \ref{Lem-a- priori estimate-semi discrete}, we get
 \begin{align}
 \label{Proof- appriori-46-a norm}
\frac{\tau ^2}{2}\norm{\CE (\delta _\tau ^ k\Bu )) }^2_ {\CA}+\norm{\CE (\Bu ^ {k+1})) }^2_ {\CA}&\le C_{G, \Bu}\frac{4 \rho ^2\tau^2 }{ \kappa}\norm{\nabla \vartheta ^ k }_{\BL^2 (\Omega)}^2+ \frac{\tau ^2 }{\kappa}\norm{\Bf^k}_{\BL ^2 (\Omega)}^2.
 \end{align}
 Then, applying \eqref{proof-apriori-57} in this estimate completes  the  proof.
 \end{proof}

\subsection{Fully discretized variational formulation} 
{Let ${\mathcal T}_h=\{T_1,\dots,T_N\}$ 
 be a quasi-uniform triangulation of 
$\Omega$  
with the  mesh width 
$h := \max_{T_i\in \mathcal{T}_h}{\rm diam}(T_i)$, where 
the elements $T_i \in \mathcal{T}_h$ are open triangles (for $d = 2$) or tetrahedra
(for $d = 3$). }
The mesh $\T_h$ is assumed to be  regular in the  sense of Ciarlet, additionally we assume the elements are 
$\gamma$-shape regular in the sense that 
we have ${\rm diam}(T_i) \le \gamma\,|T_i|^{1/{d}}$ for all $T_i\in\mathcal{T}_h$. 
{Here $|T_i|$ denotes the volume (for $d = 3$) or the area
	(for $d = 2$) of $T_i$.}
In order to provide a Galerkin discretization for \eqref{variational-elasticity.eq}-\eqref{variational-heat.eq}, we use the discrete space $\BS _0^{1,1} (\T _h) \times S ^{1,1} (\T _h)\times S ^{1,1} (\T _h) $
where
\begin{align*}
\BS ^{1,1} (\T _h) := \{ \Bu \in \BH^1 (\Omega) \qquad : \qquad \Bu | _ T \in \BP _1(T)  \quad \forall T \in \T _ h \},\\
S ^{1,1} (\T _h) := \{ u \in H^1 (\Omega) \qquad : \qquad u | _ T \in P _1(T)  \quad \forall T \in \T _ h \},
\end{align*}
and $P_1(T)$ denotes the space of polynomials of maximal degree 1 on  $T$, and $\BP_1(T):=(P_1(T))_{i=1}^d$.
We set
 $\BS_0 ^{1,1} (\T _h):= \BS ^{1,1} (\T _h)  \cap \BH _0 ^1 (\Omega)$ and $ S_0 ^{1,1} (\T _h):= S ^{1,1} (\T _h)  \cap H _0 ^1 (\Omega) $.
We define
\begin{align*}
\BV_h:=\BS_0 ^{1,1} (\T _h), \qquad W_h := S ^{1,1} (\T _h), \qquad Z_h := S ^{1,1} (\T _h).
\end{align*}
Then, we can introduce a
Ritz operator  $\BJ_h:= (\BJ_{h,\Bu}, \CJ_{h,\varphi}, \CJ_{h,\vartheta} ) : \BV \times W\times Z \rightarrow \BV_h \times W_h\times Z_h$ thorough the following equalities (see e.g., \cite[Sec. 5]{MR1770343}, \cite[Sec. 2]{MR1111453}, \cite[Sec. 2]{MR1151067} and \cite[Sec. 3]{MR958875})
\begin{align*}
 \skp{ \nabla \BJ_{h,\Bu} \Bu -\nabla \Bu , \nabla \Bv_h}&=0 \quad &&\forall \Bv_h \in \BV_h,\\
\skp{\nabla {\CJ_{h,\varphi} \varphi - \nabla \varphi}, \nabla  z_h }&=0\quad &&\forall w_h \in W_h, \\
 \skp{ \overline h \left(  \nabla \CJ _ {h,\vartheta} \vartheta - \nabla \vartheta \right) , \nabla w_h}&=0 \quad &&\forall z_h \in Z_h,
	\end{align*}
	where $ \overline h : \Omega \times I \rightarrow \R^+$ is an nonegative function.
Moreover, the entries of $\BJ _h$ satisfy the properties:
\begin{itemize}
	\item 
	The operator $\CJ _{h, \varphi }$ (and similarly  $\CJ _{h, \vartheta }$) satisfies the following  stability estimates (see e.g., \cite[Lem. 9]{MR1770343})
\begin{align*}
	\norm{\CJ_{h, \varphi } \varphi}_ {L^2 (\Omega)} &\le C _s	\norm{ \varphi}_ {L^2 (\Omega)} \qquad \forall \varphi \in L^2 (\Omega),\\
	\norm{\nabla \CJ_{h, \varphi } \varphi}_ {\BL^2 (\Omega)}& \le C ^\prime _s	\norm{\nabla  \varphi}_ {\BL^2 (\Omega)}  \qquad \forall \varphi \in H^1  (\Omega).
	\end{align*}
	\item 
	The operator $\CJ _{h, \varphi }$(and similarly  $\CJ _{h, \vartheta }$)  satisfies the following   approximation property (see e.g., \cite[Lem. 3.2]{MR1742264})
	\begin{align}
	\label{appphi}
	\norm{\varphi -\CJ_{h,\varphi} \varphi}_ {H^k (\Omega)}\lesssim h ^ {l-k}\norm{\varphi}_ {H^{l} (\Omega)} \quad k \in \{0,1\},\, l=k+1 \quad \forall \varphi \in H^l (\Omega).
	\end{align}

	\item 
	The operator $\BJ _{h, \Bu}$ satisfies the following  stability estimates \cite[Lem. 9]{MR1770343}
	\begin{align}
	\norm{\BJ_{h,\Bu} \Bu}_ {\BL^2 (\Omega)} &\le \widehat{ C}_s	\norm{ \Bu}_ {\BL^2 (\Omega)} \qquad \forall \Bu  \in \BL^2 (\Omega),\\
	\norm{\nabla \BJ_{h,\Bu} \Bu }_ {\BL^2 (\Omega)}& \le \widehat{ C}^\prime_s	\norm{\nabla  \Bu}_ {\BL^2 (\Omega )}  \qquad \forall \Bu \in \BH^1 _ 0  (\Omega),
	\end{align}
where $C_s,\, C_s ^ \prime,\, \widehat{ C}_s$, and $ \widehat{ C}^\prime_s$ are constants independent of $h$.
	\item 
	The operator $\BJ _{h, \Bu}$ satisfies the following   approximation property \cite[Lem. 9]{MR1770343}
	\begin{align}\label{appu}
	\norm{\Bu -\BJ_{h, \Bu} \Bu}_ {\BH^k (\Omega)}\lesssim h ^ {l-k}\norm{\Bu}_ {\BH^{l} (\Omega)}, \quad k \in \{0,1\},\, l=k+1 \quad \forall \Bu \in  \BH^l (\Omega).
	\end{align}
\end{itemize}
Let 
$\BBI _h : \BH^1 (\Omega)\rightarrow \BS^ {1,1}(\T _h) $ and $\CI _h :  H^1 (\Omega)\rightarrow S^ {1,1}(\T _h) $
 be Scott-Zhang operators with the  local approximation  properties 
(see e.g., \cite[Lem. 1.130]{MR3702417}):

\begin{align}\label{approx-Scott-Zhang-a}
\norm{\Bu -\BBI_h \Bu}_ {H^k (T)}&\lesssim  h ^ {l-k}\norm{\Bu}_ {\BH^{l} (\omega_T)} \quad &&k \in \{0,1\},\, l=k+1 \quad \forall \Bu \in  \BH^l (\omega _T),\\
\label{approx-Scott-Zhang-b}
	\norm{\varphi -\CI_h \varphi}_ {H^k (T)}&\lesssim h ^ {l-k}\norm{\varphi}_ {H^{l} (\omega_ T)} \quad && k \in \{0,1\},\, l=k+1 \quad \forall \varphi \in  H^l (\omega _T),
\end{align}
where $\omega _T$ is the  patch of the  element $T \in \T_h$.\\[2mm]
The fully discretized scheme for \eqref{variational-elasticity.eq}-\eqref{variational-heat.eq} reads as: For $k \in \{ 0,1,\cdots, M-1\}$, find $(\Bu ^{k+1} _ h, \varphi ^{k+1}_h , \vartheta _h ^{k+1}) \in \BV _h \times W_h \times Z_h$ such that
	\begin{subequations}
	\begin{alignat}{8}
\label{variational-elasticity-fully discrete.eq}
\nonumber
\int_\Omega \partial ^k _{\tau \tau} \Bu_h  \,\Bv _h dx&+\int_{\Omega} \left(  g(\varphi_h ^ {k+1} )+\kappa\right)  \CA \left( \CE (\Bu _h^ {k+1})\right): \CE(\Bv_h) \,dx -\rho		\int_{\Omega} \vartheta _h ^k \,\BI: \CE(\Bv_h)\,dx\\
&\quad =	 \int_{\Omega} \BBI_h \Bf ^k\cdot \Bv_h dx 
 &&\hspace{-0.9cm} \forall \Bv_h \in \BV_h,\\
\nonumber
\ell \int_{\Omega} \nabla \varphi_h ^{k+1}   \cdot \nabla w _h \, dx &+ \frac{1}{\ell} \int_{\Omega}  \varphi_h^{k+1}  \,  w _h\,dx 
{+{\gamma_0}\int_{\Omega}  [\varphi_h^{k+1}-\varphi _h^{k}]_+ \,w_h\,dx }\\&
\label{variational-phase-field-fully discrete.eq}
= \frac{-1}{\mathcal{G}_c} \int_{\Omega}g_c (\varphi_h ^ {k+1}, { \CE } (\Bu _h ^{k} )) \, w _h \,  dx  && \hspace{-0.9cm}\forall w _h\in W_h,\\[2mm]
\label{variational-heat-fully discrete.eq}
\nonumber
\int_{\Omega} \partial ^k _\tau \vartheta _h \, z_h dx &+ \int_{\Omega} K(\vartheta _h^ {k+1})\,\nabla \vartheta_h ^ {k+1} \cdot \nabla z_h dx+ \rho  \int_ \Omega\vartheta _h ^ k\nabla \cdot \delta_\tau ^ k \Bu_h \,z _hdx\\&
+ \int _ {\Gamma } \overline{\gamma} ^ {k}\, z _h ds= \int_{\Omega} \mathcal{I} _h \gamma  ^ {k}\, z _h ds  && \hspace{-0.9cm}\forall z_h \in Z_h,
	\end{alignat}
\end{subequations}
{and we set $\Bu_h ^ {-1}:= \Bu_0- \tau \Bv_0$ and  $\vartheta_h ^0 := \vartheta _0$.}
\subsubsection{A priori estimates}
In the following lemma, for $L \in \{ 1,2,\cdots, M\}$, we present a priori estimates 
for $\vartheta_h^L$ in {the} $L^2$-norm, and  $\CE (\Bu_h ^ {L})$ in 
{the} $\BL^2$- and $\CA$-norms. We also present a priori estimates for $\nabla \vartheta ^L _h$ and $\CE (\delta _\tau ^ {L-1}\Bu _h))$ with respect to  {the} $\BL^2$-norm.
\begin{lemma}(A priori estimates for $\vartheta ^L_h$, $\nabla \vartheta_h^L$, $\CE (\delta _\tau ^ {L-1}\Bu _h)$, and  $\CE (\Bu_h ^ {L}))$ for $L \in \{ 1,2,\cdots, M\}$)
	Let $ L \in \{ 1,2,\cdots, M\}$,
	\label{Lem-a- priori estimate-discrete solution}
	and  $\Bu^L_h  \in \BV _h$, $\varphi^L_h \in W_h$ and  $\vartheta^L_h \in Z_h$ be the solutions of \eqref{variational-elasticity-fully discrete.eq}, \eqref{variational-phase-field-fully discrete.eq} and  \eqref{variational-heat-fully discrete.eq}, respectively. Then, for  $\Bf^L\in \BL^2(\Omega)$, $ \overline{\gamma }^L \in L^2(\Gamma)$, and $\gamma ^L \in L^2(\Omega)$, we have the a priori estimate for $\vartheta _h ^ L$
	\begin{align}
	\nonumber
	\norm{\vartheta_h ^ {L}}^2_ {L^2 (\Omega)}+ \tau   \norm { \nabla \vartheta_h ^ {L}} ^2 _ {\BL^2 (\Omega)} &
\le C  _ {st, \vartheta} \sum_{k=0}^{L-1}\left( \tau \norm{\overline{\gamma }^k}^2_ {L^2 (\Gamma)}+\gamma _\kappa \tau \norm{\BBI _h\Bf ^ k }^{(\beta +2)^2 /(\beta +1)^2}_{\BL^2 (\Omega)} \right. \\[2mm]
	\label{eq-a- priori estimate-full discrete-heat}
&\left. \quad+\tau ^2  \norm{\CI_h \gamma ^k}^2_{L^2 (\Omega)}+\tau((\gamma _ \kappa +1)\abs{\Omega}^2)\right)
\le C^ \prime_ {st, \vartheta} \tau \mathcal{L}_{1,L},
	\end{align}	
	where $\mathcal L _ {1,L}$ and $\gamma_\kappa$ are defined in Lemma \ref{Lem-a- priori estimate-semi discrete}. The constant $C _ {st, \vartheta}$ only depends on $\Omega$, and the stability constants of $\BBI_h$ and $\CI_h$. Moreover, the constant  $C^ \prime_ {st, \vartheta} $ only depends on $\Omega$ and $C _ {st, \vartheta}$.
Then, we  have the following  estimates in the $\BL^2$-norm and the $\CA$-norm  
for  $\CE ( \Bu^L_h)$     and $\CE (\delta _\tau ^ {L-1}\Bu _h)$:
	\begin{align}
\nonumber
\frac{\tau ^2}{2}\norm{\CE (\delta _\tau ^ {L-1}\Bu _h)) }^2_ {\BL ^2 (\Omega)}&+\norm{\CE (\Bu_h ^ {L})) }^2_ {\BL ^2 (\Omega)}
\\[2mm]& \nonumber
 \le C _ {st, \CE\Bu} \left( \frac{\rho ^2\tau^2 }{ \alpha _ \kappa } \mathcal{L}_{1,L}+ \frac{\tau ^2}{2\alpha _{\kappa}}
\sum_{k=0}^{L-1}\norm{\BBI _h\Bf^k}_{\BL ^2 (\Omega)}^2\right)\\[2mm]
 &\le C^ \prime  _ {st, \CE\Bu} \left( \frac{\rho ^2\tau^2 }{ \alpha _ \kappa } \mathcal{L}_{1,L}+ \frac{\tau ^2}{2\alpha _{\kappa}}
\sum_{k=0}^{L-1}\norm{\Bf^k}_{\BL ^2 (\Omega)}^2\right) ,
\end{align} 
and
 	\begin{align}
\nonumber
\norm{\CE (\Bu _h^ {L})) }^2_ {\CA} &\le C^ \prime  _ {st, \Bu, \CA} \left( \frac{\rho ^2\tau^2 }{ \kappa } \widehat {\mathcal{L}}_{1,L}+ \frac{\tau ^2}{{\kappa}}
\sum_{k=0}^{L-1}\norm{ \Bf^k}_{\BL ^2 (\Omega)}^2\right) ,
\end{align}
where the constants $C ^ \prime _ {st, \CE \Bu}$ and  $C^ \prime  _ {st, \Bu, \CA}$ only depend on $\Omega$, and the stability constants of $\BBI_h$ and $\CI_h$. Moreover, $ {\mathcal{L}}_{1,L}$,   $\alpha _ \kappa$  and $\widehat {\mathcal{L}}_{1,L}$ are defined in Lemmas \ref{Lem-a- priori estimate-semi discrete} and \ref{Lem-a- priori estimate in A-norm}, respectively.
\end{lemma}
\begin{proof}
	The proof is done  similar  to Lemmas \ref{Lem-a- priori estimate-semi discrete} and \ref{Lem-a- priori estimate in A-norm}  based on   \eqref{variational-elasticity-fully discrete.eq}-\eqref{variational-heat-fully discrete.eq}.
\end{proof}

\section{Main results}\label{section4} 
In this section, we state and proof the main result of this article, which is to present a priori error estimates for   $\Bu ^L _h -\Bu ^L$ and $ \varphi ^L _h-{\varphi ^L}$, and $ \vartheta ^L _h-{\vartheta ^L}$ $(L\in \left\lbrace 1,2,\cdots,M\right\rbrace )$, respectively. 
We introduce the  following operators:
\begin{align*}
[\CA _1 (\varphi , \Bu): \Bv]&:=\int_{\Omega} (g (\varphi)+\kappa)\CA (\CE (\Bu ) ) : \CE(\Bv)  \,dx, \\
[\CA _2 (\varphi, \Bu): w]&:=  \int_{\Omega}g_c (\varphi, { \CE } (\Bu)) \, w \,  dx, \\
[\CA _3 (\vartheta , \Bu): z]&:=	\int_{\Omega} \vartheta \,\nabla \cdot \Bu\,z \,dx.
\end{align*}
{For $k=0,1,\cdots,M$}, we also define  the following difference functions
\newpage
\begin{align*}
\widehat \Be^k _ {\Bu,h} := \Bu ^k _h - \Bu^k, 
\qquad  
\widehat e^k _ {\varphi,h} := \varphi^k  _h -\varphi^k, 
\qquad  
\widehat e ^k _ {\vartheta,h} := \vartheta ^k  _h -\vartheta ^k.
\end{align*}

\begin{theorem}
	\label{Thm. the main result}
	Let $k \in \left\lbrace {1,\cdots,M}\right\rbrace $, and $\Bu ^k\in \BH ^ 2 (\Omega) \cap \BH^1_0(\Omega)$,  $\varphi ^k\in  H ^ 1 (\Omega)$ and $\vartheta ^k\in  H ^ 1 (\Omega)$ be the solutions of \eqref{variational-elasticity.eq}, \eqref{variational-phase-field.eq} and \eqref{variational-heat.eq}, respectively. Moreover, we consider $\Bu^k _h$, $\varphi ^k _h$ and  $\vartheta ^k  _h$ as the solutions of \eqref{variational-elasticity-fully discrete.eq}, \eqref{variational-phase-field-fully discrete.eq} and \eqref{variational-heat-fully discrete.eq}, respectively. Assume the space and time discretization parameters $h$, $\tau$ and the right hand side functions $\Bf ^k \in \BH^1 (\Omega)$, $\gamma ^k \in H^1 (\Omega)$, and $\overline{\gamma }^k\in L^2 (\Gamma)$ satisfy the assumptions:
	\begin{itemize}
		\item $\ell \left( \frac{\rho ^2\tau^2 }{ \kappa } \widehat {\mathcal{L}}_{1,M}+ \frac{\tau ^2}{{\kappa}}
		\sum_{m=0}^{M}\norm{ \Bf^m}_{\BL ^2 (\Omega)}^2\right) \le1$, \\[2mm]
		\item 
		$\ell \left( \frac{\rho \tau }{{ \alpha _ \kappa} } \sqrt{\mathcal{L}_{1,M}}+ \frac{\tau }{\sqrt{2\alpha _{\kappa}}}
		\sum_{m=0}^{M}\norm{\Bf^m}_{\BL ^2 (\Omega)}^2\right) \le 1$, \\[2mm]
		\item  $2 \tau^2 {\mathcal{L}}_{1,M} \le\kappa C _ {ell, \CA}$, \\[2mm]
		\item  $ \left( \frac{\rho \tau }{ \sqrt{\alpha _ \kappa} }\sqrt{ {\mathcal{L}}_{1,M}}+ \frac{\tau }{\sqrt{\alpha _ \kappa}}
		\sum_{m=0}^{M}\norm{ \Bf^m}_{\BL ^2 (\Omega)}\right)  \le 1$,
	\end{itemize}
where $ {\mathcal{L}}_{1,M}$,   $\alpha _ \kappa$  and $\widehat {\mathcal{L}}_{1,M}$ are defined in Lemmas \ref{Lem-a- priori estimate-semi discrete} and \ref{Lem-a- priori estimate in A-norm}, and $C_ {ell, \CA}$ is the ellipticity constant of $\CA$. Then, the following  estimate holds true for $L \in \{ 0,1,\cdots, M-1 \}$
\begin{align}
\nonumber
\norm { \right. & \left. \partial _ \tau ^ {L} \widehat \Be _ {\Bu,h}  }^2_ {\BL^2 (\Omega)}+\kappa \tau  \norm{ \CE( \widehat \Be^ {L+1} _ {\Bu,h})}^2_ {\BL ^2 (\Omega)}+\norm {\widehat \Be ^ {L+1}_ {\vartheta,h}  }^2_ {L^2 (\Omega)}+\tau \norm {\nabla \widehat \Be ^ {L+1}_ {\vartheta,h}  }^2_ {\BL^2 (\Omega)} 
+{\ell \tau }\norm{\nabla \widehat  e ^ {L+1} _ {\varphi ,h}}^2 _ {\BL^2 (\Omega)} \\[2mm]
\nonumber
& +\frac{\tau }{\ell }\norm{\widehat e ^ {L+1}_ {\varphi ,h}}^2 _ {L^2 (\Omega)} \lesssim  \mathcal{L}_{1,L+1}   \tau^2 \norm{\nabla  \vartheta ^ {L+1}}^2_ {\BL^2 (\Omega)}+ \kappa ^ {-1}h^2
\left( \norm{ \CE (\Bu^{L+1}  )}^2_ {\BH ^1 (\Omega)} \norm{\varphi ^{L+1}} ^2 _ {H^1 (\Omega)}	
+\norm{\vartheta ^{L}} ^2 _ {H^1 (\Omega)}\right)\\[2mm]& \nonumber
+h^2  \norm{\CE(\Bu^{L+1})}^2_ {\BH^ 1(\Omega)}\norm{\varphi ^{L+1}} ^2 _ {H^1 (\Omega)}+
\ell ^ {-1}\tau h^2\left( \norm{ \CE (\Bu^{L+1}  )}^2_ {\BH ^1 (\Omega)} \norm{\varphi ^{L+1}} ^2 _ {H^1 (\Omega)}\right. \\[2mm]
& \nonumber 
\left. +\norm{\varphi ^{L+1}} ^2 _ {H^1 (\Omega)}+\norm{ \CE (\Bu^{L+1}  )}^4_ {\BH ^1 (\Omega)} \right) +\ell\tau  h^2  \norm{\CE(\Bu^{L+1})}^4_ {\BL^2(\Omega)}\norm{\varphi ^{L+1}} ^2 _ {H^1 (\Omega)}\\[2mm]
\nonumber
&
+\tau h^2\left(\norm{\partial_{\tau \tau} ^L \Bu}^2_ {\BH^1(\Omega)}+  \norm{\Bf^ {L+1}}^2_{\BH ^1 (\Omega)}+\norm{\vartheta^{L+1}}^2_ {H^1(\Omega)}\right)\\[2mm]& 
\nonumber
+h^2 \tau^2 \left(\norm{\vartheta^{L}}^2_ {H^1(\Omega)} \norm{\delta ^L_{\tau }  \Bu}^2_ {\BH^2(\Omega)}+\norm{\delta ^L_{\tau } \vartheta}^2_{H^1(\Omega)}+\norm{\gamma ^ {L+1}}^2_{H^1(\Omega)}\right).   
\end{align}
	where the constant of the inequality  is independent of critical parameters such as the mesh size $h$, $\tau$, $\ell$, $\mu$, $\lambda$ and $\kappa$. Moreover, $\mathcal{L}_{1,{L+1}} $ is defined in Lemma \ref{Lem-a- priori estimate-semi discrete}.
\end{theorem}
\begin{proof}
	For $k=0,1,,\cdots, L$, we define the difference functions
	\begin{align*}
		\Be^k _ {\Bu,h} := \Bu ^k _h - \BJ _{h,\Bu}\Bu^k, 
\qquad  
e^k _ {\varphi,h} := \varphi^k  _h -\CJ _{h,\varphi}\varphi^k, 
\qquad  
e ^k _ {\vartheta,h} := \vartheta ^k  _h -\CJ _{h, \vartheta}\vartheta ^k,
	\end{align*}
then, subtracting \eqref{variational-elasticity-fully discrete.eq}-\eqref{variational-heat-fully discrete.eq} and \eqref{variational-elasticity.eq}-\eqref{variational-heat.eq} gives us
	\begin{subequations}
	\begin{alignat}{8}
		\label{error-elasticity.eq}
	\nonumber	 \skp{\partial ^k _{\tau\tau} \Be _ {\Bu,h}  , \,\Bv _h }&+
		\nonumber [\CA _1 (\varphi  ^{k+1}  _h , \Bu^{k+1} _h)-\CA _1 (\CJ _{h,\varphi}\varphi ^{k +1} , \BJ _{h, \Bu}\Bu ^{k+1}): \Bv _h] \\[2mm]
		\nonumber&=
		[\CA _1 (\varphi^{k+1}  , \Bu ^{k+1} )-\CA _1 (J_{h, \varphi}\varphi ^{k+1}  , \BJ_{h, \Bu}\Bu^{k+1} ): \Bv _h]\\[2mm]&
		\nonumber
+\skp{\partial_{\tau \tau} ^k\left( \Bu-  \BJ _{h,\Bu}\Bu\right),\,\Bv_h }	+\rho		\skp{ e^k _ {\vartheta,h} \,\BI,\, \nabla \cdot \Bv_h}
	\\[2mm]
	\nonumber
&+\rho \skp{(\CJ _ {h, \vartheta}\vartheta^k-\vartheta^k)\BI , \,\nabla \cdot \Bv_h }		\\[2mm]
&+ \skp{\BBI_h \Bf^k-\Bf^k, \, \Bv _h },&&\hspace{0.3cm}\forall \Bv_h \in \BV _h ,\\[4mm]
		\label{error-phase-field.eq}
		\nonumber \ell \skp{ \nabla e ^{k+1}_ {\varphi,h}   , \nabla w _h }& + \frac{1}{\ell} \skp{ e ^{k+1}_ {\varphi,h},   \,  w _h}  
+  \frac{1}{\mathcal{G}_c}\left(  [ \CA _2 (\varphi^ {k+1}  _h  , \Bu ^k _h)-\CA _2 (\CJ _{h, \varphi}\varphi ^ {k+1} , \BJ _{h, \Bu}\Bu ^k ): w_h]\right.  \\[2mm]&
		\nonumber\left.  -[  \CA _2 (\varphi ^ {k+1},\Bu^k  )-\CA _2 (\CJ _{h,\varphi} \varphi ^ {k+1}  , \BJ _{h, \Bu}\Bu ^k ) : w_h]\right)\\[2mm]
		\nonumber&{+{\gamma_0}\skp{[\varphi_h^{k+1}-\varphi_h ^{k}]_+-[\CJ_{h, \varphi}\varphi^{k+1}-\CJ_{h, \varphi}\varphi ^{k}]_+ ,w_h }}\\
		\nonumber
		&  ={\ell}\skp{ {\nabla (\varphi ^{k+1}-\CJ_{h, \varphi} \varphi^{k+1} ) }, \nabla w_h} 	\\[2mm]
		\nonumber
		&+\frac{1}{\ell} \skp{ {\varphi ^{k+1}-\CJ_{h, \varphi} \varphi ^{k+1} }, w_h}
		\\[2mm]
		\nonumber&{+{\gamma_0}\skp{[\varphi^{k+1}-\varphi ^{k}]_+-[ \CJ_{h, \varphi}\varphi^{k+1}-\CJ_{h, \varphi}\varphi ^{k}]_+ ,w_h }}
		, &&\,\,\forall w _h\in W _h,
		\\[4mm]
		\skp{\delta ^k _\tau e  _ {\vartheta,h}, \, z_h} & + \skp{\left( K(\vartheta _h^ {k+1})\,\nabla \vartheta_h ^ {k+1}-K(\CJ_{h, \vartheta}\vartheta ^ {k+1})\,\nabla \CJ_{h,\vartheta}\vartheta ^ {k+1}  \right),\, \nabla z_h}\\[2mm]&
	\label{error-heat.eq}
	\nonumber
+\rho \, [\CA _3 (\vartheta _h ^ k, \delta_\tau ^ k \Bu_h)-\CA _3 (\vartheta ^k, \delta_\tau ^ k \Bu): z_h]=\skp{{\delta _\tau ^k(\vartheta- \CJ_{h, \vartheta } \vartheta)}, \, z_h}\\[2mm]&
\nonumber+ \skp{\left( K(\vartheta ^ {k+1})\,\nabla \vartheta ^ {k+1}-K(\CJ_{h, \vartheta}\vartheta ^ {k+1})\,\nabla \CJ_{h,\vartheta}\vartheta ^ {k+1}  \right),\, \nabla z_h}
\\[2mm]&+ \skp{ \left( \CI _h \gamma  ^ {k}- \gamma  ^ {k}\right), \, z _h} 
 &&\hspace{0.3cm}\forall z_h \in Z_h.
 	\end{alignat}

\end{subequations}
{Thought the proof, we  consider $C$ as a generic constant independent of $h$, $\tau$, $\ell$, and $\kappa$, and may change from line to line.}	In order to simplify the proof, we split the procedure into four main  steps:\\
	{\bf Step 1:}  \textbf{The displacement error equation}.
	 In this step, we deal with the terms that appear in \eqref{error-elasticity.eq}. For the sake of simplicity,  this step is divided into {three} sub steps.\\	
%
	{\bf Step 1.1:}
We use the test function $\Bv _h=\delta _ \tau ^ {k}\Be _ {\Bu,h} $
in the variational formulation \eqref{error-elasticity.eq}. 
We start with rewriting  the operator  $\CA _1$ as follows
	\begin{align}
\nonumber
&[\CA _1 (\varphi ^ {k+1   } _h , \Bu^ {k+1   } _h)-\CA _1 (\CJ _{h, \varphi}\varphi^ {k+1   }   , \BJ _{h, \Bu}\Bu ^ {k+1   } ): \CE( \delta _ \tau ^k\Be _ {\Bu,h} )] \nonumber \\
&= \skp{(\kappa+(\CJ_{h, \varphi} \varphi ^ {k+1   } )^2)\CA (\CE(  \Be^ {k+1}  _ {\Bu,h} )), \CE(\delta _ \tau ^k\Be _ {\Bu,h})} \nonumber \\
\nonumber
&\quad+ 2 \skp{e ^ {k+1   } _ {\varphi,h}\CJ _{h,\varphi} \, \varphi ^ {k+1   } \, \CA (\CE ( \Bu _h^ {k+1   }))   ,\CE(\delta _ \tau ^k\Be _ {\Bu,h})} \\
\nonumber
&\quad+ \skp{(e ^ {k+1   } _ {\varphi,h})^2 \,\CA (\CE (\BJ _{h, \Bu} \Bu ^ {k+1   })) , \CE(\delta _ \tau ^k\Be _ {\Bu,h})}
\\
\nonumber
&\quad+ \skp{(e ^ {k+1   }_ {\varphi,h})^2 \, \CA (\CE (\Be^ {k+1   } _{\Bu ,h})), \CE(\delta _ \tau ^k\Be _ {\Bu,h}) )}\\
\label{upper bound-A1}
&=: T _1+T _2 +T _3+T_4.
%
%
\end{align}
Using the ellipticity property of the operator $\CA$ from Lemma \ref{Lem-holder contin}, we get
\begin{align}
\nonumber
T_1 &\ge \frac{\kappa \tau C _ {ell, \CA}}{2}\norm{ \CE(\delta _ \tau ^k\Be _ {\Bu,h)}}^2_ {\BL ^2 (\Omega)}+\frac{\kappa  {C_{ell, \CA}}}{2\tau}\left( \norm{\CE (\Be ^ {k+1} _ {\Bu,h})}^2_ {\BL ^2 (\Omega)}-\norm{ \CE(\Be^ {k-1} _ {\Bu,h})}^2_ {\BL ^2 (\Omega)}\right)\\
\label{Proof-eq 2.41}
&\quad+
\frac{1}{4 \tau}\left( \norm{\CJ _ {h, \varphi} \varphi ^ {k+1}\CA^ {1/2}(\CE (\Be ^ {k+1} _ {\Bu,h}))}^2_ {\BL ^2 (\Omega)}-\norm{\CJ _ {h, \varphi} \varphi ^ {k+1}\CA^ {1/2}(\CE (\Be ^ {k-1} _ {\Bu,h}))}^2_ {\BL ^2 (\Omega)}\right),
\end{align}
where $C _ {ell,\CA}$ is the ellipticity constant of $\CA$.
In order to find an upper bound for $T_2$, we apply the   the H\"older inequality \cite[Prop. II.2.18]{MR2986590},    the Sobolev embedding
	  $W^ {1,4} (\Omega ) \subset L^\infty (\Omega) $ from \cite[Thm. 1.20]{MR3014456}, and  we note that  $ \Bu_h ^ {k+1}$ is indeed  in  $  \mathbf S_0^ {1,1}(\T_h) $ (which results in  $\abs{\CE(\Bu _h ^{k+1})}_ {1,4,T}=0$) to  get
{	\begin{align}
\nonumber
\frac{1}{2}\abs{T _2} &\le \norm{e ^ {k+1   }_ {\varphi,h}}_ {L^2 (\Omega)} \norm{\CA^ {1/2} (\CE ( \Bu_h ^ {k+1   } ))}_ {\BL ^ \infty (\Omega)}\norm{\CJ _ {h , \varphi}\varphi ^ {k+1   } \CA^ {1/2}(\CE(\delta _\tau ^k\Be _ {\Bu,h} ))}_ {\BL^2 (\Omega)}\\[2mm]
\nonumber
& \le \norm{e ^ {k+1   }_ {\varphi,h}}_ {L^2 (\Omega)} \sum_{T \in \T_h}\norm{\CA (\CE ( \Bu_h ^ {k+1   } ))}_ {\BL^4 (T)}\norm{\CJ _ {h , \varphi}\varphi ^ {k+1   }\CE(\delta _\tau ^k\Be _ {\Bu,h} )}_ {\BL^2 (\Omega)}
\\
\nonumber
&\le 2 \ell \norm{\CA^ {1/2} (\CE ( \Bu_h ^ {k+1   } ))}^2_ {\BL ^2 (\Omega)} \norm{\CJ _ {h , \varphi}\varphi ^ {k+1   }\CA^ {1/2} (\CE(\delta _\tau ^k \Be  _ {\Bu,h}) )}^2_ {\BL ^2 (\Omega)}\\
\label{Proof-eq 2.42} 
&\quad+\frac{1}{8 \ell} \norm{e ^ {k+1   }_ {\varphi,h} }^2_ {L^ 2 (\Omega)},
\end{align}}
 where the last line of \eqref{Proof-eq 2.42} is estimated by using  \eqref{Gaglias}, and again noting  that $ \Bu_h ^ {k+1}|_T \in \BP_1(T)$ for every $T \in \T_h$ (which results in $\abs{\CE (\Bu^{k+1}_h)}_{\BH^1(T)}=0$), and by Young's inequality.
Then,  from  the a priori estimate from Lemma \ref{Lem-a- priori estimate-discrete solution}, we exploit that 
{ \begin{align}
\abs{T_2}&\le 
\label{Proof-eq 2.43} 
\nonumber
4 \ell C^ \prime  _ {st, \Bu, \CA}\left( \frac{\rho ^2\tau^2 }{ \kappa } \widehat {\mathcal{L}}_{1,k+1}+ \frac{\tau ^2}{{\kappa}}
\sum_{m=0}^{k}\norm{ \Bf^m}_{\BL ^2 (\Omega)}^2\right) \norm{\CJ _ {h , \varphi}\varphi ^ {k+1   }\CA^ {1/2} (\CE(\delta _\tau ^k \Be _ {\Bu,h} ))}^2_ {\BL ^2 (\Omega)}\\
& \quad + \frac{1}{4 \ell}\norm{e^ {k+1   } _{\varphi,h}}^2_ {L ^2 (\Omega)}.
\end{align}}
To control the term $T_3$, we use  
	 the H\"older inequality \cite[Prop. II.2.18]{MR2986590},  the Sobolev embedding  $W^ {1,4} (\Omega ) \subset L^\infty (\Omega) $, and take into account  that $\BJ _ {h, \Bu} \Bu ^ {k+1}|_T \in \boldsymbol \BP  _1(T) $ for all $T \in \T _h$, and apply \eqref{Gaglias}  with Young's inequality to get
	\begin{align}
		\nonumber
	\abs{	T_3}&\le \norm{ e ^ {k+1} _ {\varphi,h}\CA ^ {1/2} \left( \CE(\delta _ \tau ^k\Be _ {\Bu,h}) \right) }_{\BL^2 (\Omega)}\norm{ e ^ {k+1} _ {\varphi,h}}_{L^2 (\Omega)}\norm{\CA ^ {1/2} (\CE (\BJ _{h, \Bu} \Bu ^ {k+1} )) }_{\BL^\infty (\Omega)}
		\\& \nonumber
		\le \norm{ e ^ {k+1} _ {\varphi,h}\CA ^ {1/2}(\CE(\delta _ \tau ^k\Be _ {\Bu,h})) }_{\BL^2 (\Omega)}\norm{ e ^ {k+1} _ {\varphi,h}}_{L^2 (\Omega)}\sum_{T \in \T_h}\norm{\CA ^ {1/2} (\CE (\BJ _{h, \Bu} \Bu ^ {k+1} )) }_{\BL^4 ( T)}
%
		\\
		&\nonumber\le 2 \ell\norm{\CA ^ {1/2} (\CE (\BJ _{h,\Bu} \Bu ^ {k+1}))-\CA ^ {1/2}(\CE (\Bu ^ {k+1} ))}^2_ {\BL ^2 (\Omega)} \norm{e ^ {k+1}_{\varphi,h}\,\CA ^ {1/2}(\CE(\delta _ \tau ^k\Be _ {\Bu,h}))}^2_ {\BL ^2 (\Omega)}\\
		\nonumber
		&\quad +2 \ell \norm{\CA ^ {1/2}(\CE (\Bu  ^ {k+1}))}^2_ {\BL ^2 (\Omega)} \norm{e^ {k+1}_{\varphi,h}\,\CA ^ {1/2}(\CE(\delta _ \tau ^k\Be _ {\Bu,h}))}^2_ {\BL ^2 (\Omega)}+\frac{1}{8 \ell} \norm{e^ {k+1}_{\varphi,h}}^2_ {L ^2 (\Omega)}.
	\end{align}
Applying 
	the a priori estimate from Lemma \ref{Lem-a- priori estimate in A-norm}, and  stability properties of the operators $\BJ _{h, \Bu}$, lead to
	\begin{align}
	\label{Proof-eq 2.44}
	\nonumber 
\abs{T _3} &\le 	4 \ell C _ {st, \Bu, \CA} \left( \frac{\rho ^2\tau^2 }{ \kappa } \widehat {\mathcal{L}}_{1,k+1}+ \frac{\tau ^2}{{\kappa}}
	\sum_{m=0}^{k}\norm{ \Bf^k}_{\BL ^2 (\Omega)}^2\right) \norm{e^ {k+1} _ { \varphi,h}\CA ^ {1/2}(\CE(\delta _\tau ^k \Be  _ {\Bu,h} ))}^2_ {\BL ^2 (\Omega)}\\
	& \quad + \frac{1}{4 \ell}\norm{e^ {k+1   } _{\varphi,h}}^2_ {L ^2 (\Omega)}.
	\end{align}	
Finally,  the term $ {T} _4$ can be rewritten in the following form 
	\begin{align}
	\nonumber
		\label{Proof-eq 2.48a} 
		{T} _4&= \frac{\lambda}{\tau}\norm{e  ^ {k+1}_ {\varphi ,h} \,\nabla \cdot \delta_\tau ^ k\Be  _ {\Bu,h} }^2_ {L ^2 (\Omega)}+ \frac{\mu}{\tau}	\norm{e ^ {k+1} _ {\varphi ,h}  \CE (\delta^k_\tau \Be _ {\Bu,h} )}^2_{\BL ^2(\Omega)}\\&
		\quad +\frac{\lambda}{\tau}\skp{e  ^ {k+1}_ {\varphi ,h} \,\nabla \cdot \Be ^ {k-1} _ {\Bu,h}, \delta_\tau ^ k\Be  _ {\Bu,h} }+ \frac{\mu}{\tau}	\skp{e ^ {k+1} _ {\varphi ,h}  \CE (\Be ^ {k-1} _ {\Bu,h} ),\delta_\tau ^ k\Be  _ {\Bu,h}} .
	\end{align}
	{\bf Step 1.2:} In this step, we deal with the first term on the right-hand side of \eqref{error-elasticity.eq}.
	Applying the test function  $\Bv _h=\delta_\tau ^k\Be _ {\Bu,h} $, we have 
	\begin{align}\nonumber
	&[  \CA _1 (\varphi ^ {k+1}  , \Bu ^ {k+1} )-\CA _1 (\CJ_{h, \varphi}\varphi ^ {k+1}  , \BJ_{h,\Bu}\Bu^ {k+1}  ) : \CE(\delta _\tau ^k\Be _ {\Bu,h} )]\\
&=\kappa\int_{\Omega}\left(\CA(\CE (\Bu^ {k+1} ))-\CA (\CE (\BJ_{h,\Bu} \Bu ^ {k+1} )) \right)  : \CE(\delta ^k _\tau \Be _ {\Bu,h} ) \,dx 
\nonumber	\\& \nonumber +  \int_{\Omega}\left(   (\varphi ^ {k+1} )^2- (\CJ_{h, \varphi} \varphi ^ {k+1} )^2\right)\CA(\CE (\Bu ^ {k+1} ))  : \CE(\delta ^ k _ \tau\Be _ {\Bu,h} ) \,dx\\
&\nonumber
	+\int_{\Omega}   (\CJ _ {h, \varphi} \varphi ^ {k+1} )^2\left(\CA(\CE (\Bu ^ {k+1} ))-\CA (\CE (\BJ_{h,\Bu} \Bu ^ {k+1} )) \right)  : \CE( \delta_\tau ^k\Be  _ {\Bu,h} ) \,dx
	\\
	&
	\label{Proof-main-step3-eq1}
	:=T_5+T_6+T_7.
	\end{align}
It follows from the Lipschitz continuity of the operator $\CA$ from Lemma \ref{Lem-holder contin} and  approximation properties of 
	the operator  $\BJ_{h, \Bu}$ that
 	\begin{align}
	\nonumber
\abs{	T_5} & \le \frac{4\kappa C _ {\mu,\lambda}}{\tau C_ {ell, \CA}} \norm{\CE (\Bu ^ {k+1} )-\CE (\BJ_{h, \Bu} \Bu ^ {k+1} )}_{\BL^2(\Omega)}^2+\frac{\kappa \tau C_{ell, \CA}}{8}\norm{\CE(\delta_\tau ^k \Be _ {\Bu,h} ) }_{\BL^2(\Omega)}^2
	\\&\label{Proof-eq 2.55lm}
	\le \frac{4\kappa h^2 C C _ {\mu,\lambda}}{C_ {ell, \CA}\tau } \norm{\CE (\Bu ^ {k+1} )}_{\BH^1(\Omega)}^2+\frac{\kappa \tau C_{ell, \CA}}{8}\norm{\CE(\delta_\tau ^k \Be _ {\Bu,h} ) }_{\BL^2(\Omega)}^2.
	\end{align}	
	Applying the H\"older inequality \cite[Prop. II.2.18]{MR2986590},   the Sobolev embedding  $W^ {1,4} (\Omega ) \subset L^\infty (\Omega) $, as well as taking into account that for $T\in \T_h$, we have $\Be ^k _ {\Bu,h} |_ T \in \BP_1(T)$ which results in  $\abs{\CE (\delta ^k _ \tau \Be _ {\Bu,h})}_{1,4,T}=0$,  applying  \eqref{Gaglias}, and  employing 
	the approximation properties of $\CJ_{h,\varphi}$, we deduce that
	\begin{align}
	\nonumber
\abs{	T_6} &\le \norm{(\varphi  ^{k+1})^2 - (\CJ _{h, \varphi} \varphi  ^{k+1})^2 }_ {L^2(\Omega)}  \norm{ \CA (\CE ( \Bu    ^{k+1} ))}_{\BL ^2(\Omega)}\sum_{T \in \T_h}\norm{\CE(\delta _\tau ^k \Be _ {\Bu,h} ) } _ {\BL^4( T)}\\
	\nonumber
	&\nonumber\le C \norm{\varphi ^{k+1}- \CJ _{h, \varphi} \varphi ^{k+1}}_ {L^2(\Omega)} \norm{ \CA (\CE (\Bu ^{k+1} ))}_{\BL ^2 (\Omega)}\norm{\CE(\delta _ \tau ^k\Be _ {\Bu,h} ) } _ {\BL ^2 (\Omega)}\\
	& \nonumber\le 4CC^{-1} _{ell, \CA}\kappa ^ {-1}h^2\tau ^{-1}\norm{\CA( \CE (\Bu^{k+1}  ))}^2_ {\BL ^2 (\Omega)}\norm{\varphi ^{k+1}} ^2 _ {H^1 (\Omega)}\\& \label{Proof-eq 2.55} 
	\quad +\frac{\kappa C _ {ell, \CA}}{8\tau}\left( \norm{\CE (\Be ^ {k+1} _ {\Bu,h})}^2_ {\BL ^2 (\Omega)}+\norm{ \CE(\Be^ {k-1} _ {\Bu,h})}^2_ {\BL ^2 (\Omega)}\right),
	\end{align}
	Using the Sobolev embedding   $W^ {1,4} (\Omega ) \subset L^\infty (\Omega) $  as well as taking into account that for all $T\in \T_h$, $\CJ _{h, \varphi}\varphi ^ {k+1}(T)$ is in fact a linear polynomial, using \eqref{Gaglias}, applying  
		Lemma \ref{Lem-holder contin} and the approximation properties of 
	the operator  $\BJ_{h, \Bu}$, the following estimate can be obtained  similar to \eqref{Proof-eq 2.55lm}
 {	\begin{align}
	\nonumber 
\abs{T_7} &\le 
 \norm{\CJ _{h, \varphi}\varphi ^ {k+1} }_ {L^ \infty(\Omega)}\abs{\skp{\CA ^ {1/2}(\CE (\Bu ^ {k+1})-\CE (\BJ_{h, \Bu}\Bu ^ {k+1})),\CJ _{h, \varphi}\varphi ^ {k+1}\CA ^{1/2}(\CE(\delta_\tau ^k\Be _ {\Bu,h} )) }}\\&
		\label{Proof-eq 2.56}
			\le {4 CC_ {\mu,\lambda}\tau ^ {-1}h^2 } \norm{\CJ _{h, \varphi}\varphi ^ {k+1} }_ {H^1 (\Omega)} ^2  \norm{\CE (\Bu ^ {k+1} )}_{\BH^1(\Omega)}^2+\frac{\tau}{8}\norm{\CJ _{h, \varphi}\varphi ^ {k+1}\CA ^ {1/2}(\CE(\delta_\tau ^k \Be _ {\Bu,h} )) }_{\BL^2(\Omega)}^2.
	\end{align}}
By combining \eqref{Proof-main-step3-eq1}-\eqref{Proof-eq 2.56}, we obtain immediately that
	\begin{align}
	\nonumber 
	\sum _ {i=5}^7 \abs{T_i}
&
	\le \frac{4\kappa h^2 C_\Bu C _ {\mu,\lambda}}{C_ {ell, \CA}\tau } \norm{\CE (\Bu ^ {k+1} )}_{\BH^1(\Omega)}^2\\& \nonumber 
	\quad +{4 CC_ {\mu,\lambda}\tau ^ {-1}h^2 } \norm{\CJ _{h, \varphi}\varphi ^ {k+1} }_ {H^1 (\Omega)} ^2  \norm{\CE (\Bu ^ {k+1} )}_{\BH^1(\Omega)}^2\\
	& \nonumber 
\quad	+4CC^{-1} _{ell, \CA}\kappa ^ {-1}h^2\tau ^{-1}\norm{\CA( \CE (\Bu^{k+1}  ))}^2_ {\BL ^2 (\Omega)}\norm{\varphi ^{k+1}} ^2 _ {H^1 (\Omega)}
	\\&
	\nonumber
	\quad +	\frac{\kappa C _ {ell, \CA}}{4\tau}\left( \norm{\CE (\Be ^ {k+1} _ {\Bu,h})}^2_ {\BL ^2 (\Omega)}+\norm{ \CE(\Be^ {k-1} _ {\Bu,h})}^2_ {\BL ^2 (\Omega)}\right)\\
	&
\quad
 \label{proof-main-step3}+\frac{\tau}{8}\norm{\CJ _{h, \varphi}\varphi ^ {k+1}\CA ^ {1/2}(\CE(\delta_\tau ^k \Be _ {\Bu,h} )) }_{\BL^2(\Omega)}^2.
	\end{align}
\\
{{\bf Step 1.3:} In this step, after substituting $\Bv _h=\delta _ \tau ^ {k}\Be _ {\Bu,h} $ as the test function, we focus on the rest of the term in 
in the variational formulation \eqref{error-elasticity.eq}.} Recalling the approximation properties of $\BJ_{h, \Bu}$ and Young's inequality, the following upper bound is
valid
\begin{align}
\nonumber
T_5:=\skp{\partial_{\tau \tau} ^k\left( \Bu-  \BJ _{h,\Bu}\Bu\right),\delta _ \tau ^ {k}\Be _ {\Bu,h} } &\le C h ^2 \norm{\partial_{\tau \tau} ^k \Bu}^2_{\BH^1 (\Omega)}\\& 
\label{EQ-K1}
\quad
+\frac{1}{8}\left(\norm {\partial _ \tau ^ {k} \Be _ {\Bu,h}  }^2_ {\BL^2 (\Omega)}\right. +\left. \norm {\partial _ \tau ^ {k-1}\Be _ {\Bu,h}}^2_ {\BL^2 (\Omega)}\right),
\end{align}
 Integration by parts  and Young's inequality lead to
 \begin{align}
 \nonumber
T_8&:=\rho		\skp{ e^k _ {\vartheta,h} \,\BI,\, \nabla \cdot \delta _ \tau ^ {k}\Be _ {\Bu,h}}  = \rho \skp{ \nabla e^k _ {\vartheta,h} \,\BI,\, \delta _ \tau ^ {k}\Be _ {\Bu,h}} \\
& \label{EQ-K2}
\le \frac{c_0 \rho}{2}\norm{\nabla e^k _ {\vartheta,h} }^2_ {\BL^2(\Omega)}
+\frac{\rho}{8c_0 }\left(\norm {\partial _ \tau ^ {k} \Be _ {\Bu,h}  }^2_ {\BL^2 (\Omega)}\right. +\left. \norm {\partial _ \tau ^ {k-1}\Be _ {\Bu,h}}^2_ {\BL^2 (\Omega)}\right).
 \end{align}
Applying Korn’s inequality \cite{horgan1983inequalities}, Young's inequality, in addition to the approximation property of $\CJ_{h , \vartheta}$, we obtain
\begin{align}
\nonumber
T_{10}&:=\rho \skp{(\CJ _ {h, \vartheta}\vartheta^k-\vartheta^k)\BI , \,\nabla \cdot \delta _ \tau ^ {k}\Be _ {\Bu,h} } 
\le \frac{CC^2_K\rho ^2 h^2}{\kappa \tau}\norm{\vartheta ^k}^2_ {H^1 (\Omega)}
 \\&\quad
 \label{EQ_K3}
 	+\frac{\kappa \tau C_{ell, \CA}}{8}\norm{\CE(\delta_\tau ^k \Be _ {\Bu,h} ) }_{\BL^2(\Omega)}^2.
\end{align}
On the other hand,	using the approximation property of $\BBI _h$,  there holds the following estimate
\begin{align}
\nonumber
T_{11}&:=\skp{\BBI _h \Bf ^ {k+1} - \Bf^ {k+1},\delta ^k _\tau \Be _ {\Bu ,h}}\le \frac{1}{2}\,\norm{\BBI _h \Bf^ {k+1} - \Bf^ {k+1}}^2_ {\BL^2 (\Omega)}+\frac{1 }{2}\norm{\delta ^k _\tau \Be _ {\Bu ,h}}^2 _ {\BL^2 (\Omega)}
\\&
\label{Proof-eq 2.63} 
\le C\,h^2 \norm{\Bf^ {k+1}}^2_{\BH ^1 (\Omega)}
+\frac{1}{8}\left(\norm {\partial _ \tau ^ {k} \Be _ {\Bu,h}  }^2_ {\BL^2 (\Omega)}\right. +\left. \norm {\partial _ \tau ^ {k-1}\Be _ {\Bu,h}}^2_ {\BL^2 (\Omega)}\right).
\end{align}
	{\bf Step 2:} \textbf{The elliptic error equation} In this step, we are concerned with the terms appearing in \eqref{error-phase-field.eq}. This step is separated into {four} sub-steps.\\
	{\bf Step 2.1:}
	We set   $w _h=e  ^ {k+1}_ {\varphi,h} $ as the test function in the variational formulation \eqref{error-phase-field.eq}, and  control the operator   $\CA _2$. 
		We note that
				\begin{align}
		\label{difference-A2}
		\nonumber \CA _2 (\varphi  ^ {k+1} _h , \Bu ^ {k} _h)-\CA _2 (\CJ _{h, \varphi}\varphi   ^ {k+1}  , \BJ _{h, \Bu}  ^ {k}\Bu )&=e  ^ {k+1}_ {\varphi ,h}  \left( \mathcal B (\CE ( \Bu   ^ {k}_h))-\mathcal B (\CE (\BJ _{h, \Bu} \Bu  ^ {k} )) \right)\\
		&\quad \nonumber+e^ {k+1} _ { \varphi, h  }\mathcal B (\CE (\BJ _ {h, \Bu}\Bu ^k )) \\
		&\quad
		+\CJ _{h, \varphi}\varphi ^ {k+1} \left(\mathcal B (\CE ( \Bu _h ^ {k}  ))-\mathcal B (\CE (\BJ _ {h, \Bu}\Bu   ^ {k}  ))\right).
		\end{align}
		Applying  the mean value theorem, the first term on the right-hand side of \eqref{upper bound-A1}  can be written in the following form
		\begin{align}
			\label{T1-A2}
			\nonumber\mathcal B (\CE ( \Bu ^k  _h))&-\mathcal B (\CE (\BJ _{h, \Bu} \Bu  ^k ))=\int_0 ^1\left(  \frac{\partial}{\partial \CE}  \mathcal{B} (\Bc)| _ {\Bc=\CE(\BJ _{h, \Bu}\Bu  ^ {k} )+ \widetilde{\rho}(\CE (\Be  ^ {k}_ {\Bu,h} ))}\right) d \widetilde{\rho}\, :\CE (\Be  ^ {k} _ {\Bu,h} )\\
			\nonumber &=\lambda  \left( \operatorname{tr} (\CE(\BJ _{h, \Bu}\Bu ^k )) \BI : \CE (\Be ^ {k} _ {\Bu,h} ) \right) +\mu  \left(\CE(\BJ _{h, \Bu}\Bu  ^ {k}):\CE (\Be  ^ {k}_ {\Bu,h} ) \right) \\
			\nonumber
			 & \quad +\frac{\lambda}{2}  \left( \operatorname{tr} (\CE (\Be  ^ {k}_ {\Bu,h} )) \BI : \CE (\Be  ^ {k}_ {\Bu,h} ) \right) + \frac{\mu}{2}  \left(\CE (\Be   ^ {k}_ {\Bu,h} ):\CE (\Be  ^ {k} _ {\Bu,h} ) \right) \\
			 &=(\CA (\CE (\BJ _ {h, \Bu}\Bu ^k)): \CE (\Be^k_{\Bu,h}))+\frac{1}{2}(\CA (\CE (\Be^k _ {h, \Bu})): \CE (\Be^k_{\Bu,h})).
	\end{align}
		Substituting \eqref{T1-A2} into \eqref{difference-A2}, we arrive at
		\begin{align}
		\nonumber
			[\CA _2 (\varphi  ^ {k+1} _h , \Bu  ^ {k}_h)&-\CA _2 (\CJ _{h,\Bu}\varphi ^ {k+1}   , \BJ  _{h,\Bu}\Bu ^ {k} ): e ^ {k+1} _ {\varphi ,h} ]\\
			& \nonumber =	\skp{e  ^ {k+1} _ {\varphi ,h}  {\left( \mathcal A (\CE (\BJ  _{h,\Bu} \Bu  ^ {k} ):\CE (\BJ _ {h,\Bu} \Bu  ^ {k} )\right) ),e  ^ {k+1} _ {\varphi ,h} }}\\& 
			\label{Proof-eq 2.47} \quad+  {\widehat{T}_1+ \widehat{T}_2
			+ \widehat{T}_3+ \widehat{T}_4} ,
		\end{align}
		where 
		\begin{align}
			\nonumber \widehat{T} _1& := \frac{\lambda}{2}  \int _ \Omega \left( \left( \operatorname{tr} (\CE (\Be  ^ {k} _ {\Bu,h} )) \BI : \CE (\Be  ^ {k} _ {\Bu,h} ) \right) + \frac{\mu}{2}  \left(\CE (\Be  ^ {k} _ {\Bu,h} ):\CE (\Be  ^ {k} _ {\Bu,h} ) \right) \right)  (e  ^ {k+1} _ {\varphi ,h} )^2 \, dx,\\
			\nonumber \widehat{T}_2&:=\skp{ e  ^ {k+1} _ {\varphi ,h}  \left( \CA (\CE(\BJ _{h, \Bu}\Bu  ^ {k}) : \CE (\Be  ^ {k}_ {\Bu,h} ) \right),e ^ {k+1} _ {\varphi ,h}},\\
			\nonumber \widehat{T}_3&:=\skp{\CJ  _{h, \varphi}\varphi ^ {k+1}  \left( \CA (\CE(\Be^k _{h, \Bu}) : \CE (\Be  ^ {k}_ {\Bu,h} )\right) ,e ^ {k+1} _ {\varphi ,h} },\\
				\nonumber \widehat{T}_4&:=\skp{\CJ  _{h, \varphi}\varphi ^ {k+1}  \left( \CA (\CE(\BJ _{h, \Bu}\Bu  ^ {k}) : \CE (\Be  ^ {k}_ {\Bu,h} )\right) ,e ^ {k+1} _ {\varphi ,h} }.
		\end{align}
		The term $ \widehat{T} _1$ can be rewritten in the following form 
		\begin{align}
			\label{Proof-eq 2.48} 
			\widehat{T} _1= \frac{1}{2}	\norm{e ^ {k+1}  _ {\varphi ,h} \CA ^ {1/2}  (\CE (\Be ^k _ {\Bu,h} ))}^2_{\BL ^2(\Omega)}. 
		\end{align}	
We note that
		\begin{align}
			\nonumber \widehat{T}_2=&\skp{e  ^ {k+1} _ {\varphi ,h} {\left( \mathcal A (\CE (\BJ _{h,\Bu} \Bu  ^ {k} ):\CE ( \Bu_h  ^ {k} )\right) ,e  ^ {k+1} _ {\varphi ,h} }}\\
			\nonumber
			& \quad-\skp{e  ^ {k+1} _ {\varphi ,h} {\left( \mathcal A (\CE (\BJ _{h,\Bu} \Bu  ^ {k} ):\CE (\BJ _{h,\Bu}\Bu  ^ {k} )\right) ,e  ^ {k+1} _ {\varphi ,h} }}\\
			\label{Proof-eq 2.49a} 
		=:&	 \widehat T_ {2,1}-\skp{e  ^ {k+1} _ {\varphi ,h} {\left( \mathcal A (\CE (\BJ _{h,\Bu}\Bu  ^ {k} ):\CE (\BJ _{h,\Bu} \Bu  ^ {k} )\right) ,e  ^ {k+1} _ {\varphi ,h} }}.
		\end{align}
			Applying the Cauchy-Schwarz and Young inequalities as well as 
		the H\"older inequality \cite[Prop. II.2.18]{MR2986590},  the Sobolev embedding  $W^ {1,4} (\Omega ) \subset L^\infty (\Omega)$, and considering that $\BJ _ {h, \Bu} \Bu ^ {k+1}$ and $\Bu _h ^k$ are piecewise linear functions,  the estimate \eqref{Gaglias}, the stability property of $\BJ_ {h, \Bu}$, and Lemmas \ref{Lem-a- priori estimate in A-norm} and \ref{Lem-a- priori estimate-discrete solution} lead to
		\begin{align}
		\nonumber
	\widehat	T_ {2,1}	&
					\nonumber\le\norm{\CA(\CE(\BJ _{h, \Bu}\Bu ^k))}_{\BL ^\infty (\Omega)}\norm{\CE (\Bu ^k _ {h} ) }_{\BL ^\infty(\Omega)}\norm{e^ {k+1} _ {\varphi ,h}  }^2_{L ^2(\Omega)}
					\\& 
					\nonumber\le \left( \sum_{T \in \T_h}\norm{\CA(\CE(\BJ _{h, \Bu}\Bu ^k))}_{\BL^4 (T)}+\sum_{T \in \T_h}\norm{\CE(\Bu_h ^k)}_{\BL^4 ( T)}\right)\norm{e^ {k+1} _ {\varphi ,h}  }^2_{L ^2(\Omega)}\\
					\nonumber
					& \le  \left( \norm{\CA(\CE(\BJ _{h, \Bu}\Bu ^k))}_{\BL ^2 (\Omega)}+\norm{\CE(\Bu_h ^k)}_{\BL ^2 (\Omega)}\right)\norm{e^ {k+1} _ {\varphi ,h}  }^2_{L ^2(\Omega)}\\
					&
				\nonumber
				\le 	\left({C _ {st, \Bu, \CA}} \left( \frac{\rho ^2\tau^2 }{ \kappa } \widehat {\mathcal{L}}_{1,k}+ \frac{\tau ^2}{{\kappa}}
					\sum_{m=0}^{k-2}\norm{ \Bf^m}_{\BL ^2 (\Omega)}^2\right)\right.\\
						\label{Proof-eq 2.49} 
					& \left.\quad \quad\quad+\sqrt{ C^ \prime  _ {st, \CE\Bu} }\left( \frac{\rho \tau }{ \alpha _ \kappa }\sqrt{ \mathcal{L}_{1,k}}+ \frac{\tau}{\sqrt{2\alpha _{\kappa}}}
					\sum_{m=0}^{k-2}\norm{\Bf^m}_{\BL ^2 (\Omega)}\right) \right) \norm{e^ {k+1} _ {\varphi ,h}  }^2_{L ^2(\Omega)}.
		\end{align}
			With the same arguments, we obtain that 
		\begin{align}
		\nonumber
\abs{\widehat T_4} &\le \norm{\CA^{1/2}(\CE(\BJ _{h, \Bu}\Bu ^k))}_{\BL ^2 (\Omega)}\norm{e^ {k+1} _ {\varphi ,h}  }_{L ^2(\Omega)}\norm{\CJ _ {h, \varphi}\varphi ^ {k+1   } \CA ^ {1/2}  (\CE ( \Be_{\Bu,h} ^ {k  } ))}_ {\BL^2 (\Omega)}\\
\nonumber
&
\nonumber 
\le {C _ {st, \Bu, \CA}} \left( \frac{\rho ^2\tau^2 }{ \kappa } \widehat {\mathcal{L}}_{1,k}+ \frac{\tau ^2}{{\kappa}}
\sum_{m=0}^{k-2}\norm{ \Bf^m}_{\BL ^2 (\Omega)}^2\right)\norm{e^ {k+1} _ {\varphi ,h}  }^2_{L ^2(\Omega)}\\
&
\quad+\norm{\CJ _ {h, \varphi}\varphi ^ {k+1   } \CA ^ {1/2}  (\CE ( \Be_{\Bu,h} ^ {k  } ))}_ {\BL^2 (\Omega)}^2.
		\end{align}
It follows from the    Cauchy-Schwarz  and Young inequalities that
		\begin{align}
			\nonumber \abs{\widehat{T}_3}&\le  \norm{\CJ _ {h , \varphi}\varphi ^ {k+1   } \CA ^ {1/2} (\CE ( \Be_{\Bu,h} ^ {k  } ))}_ {\BL ^ 2 (\Omega)}\norm{e  ^ {k+1}_ {\varphi ,h}\CA ^ {1/2}  (\CE(\Be^k _ {\Bu,h} ))}_ {\BL^2 (\Omega)}\\
			\label{Proof-eq 2.49m} 
		& 	 
			\le 4\norm{\CJ _ {h, \varphi}\varphi ^ {k+1   } \CA ^ {1/2}  (\CE ( \Be_{\Bu,h} ^ {k  } ))}^2_ {\BL^2 (\Omega)}+\frac{1}{8}\norm{e ^ {k+1   }_ {\varphi,h} \CA ^ {1/2} (\CE(\Be^k _ {\Bu,h} ))}^2_ {\BL^2 (\Omega)}.
		\end{align}
		{\bf Step 2.2:} 
		Choosing  $w _h=e ^ {k+1}  _ {\varphi,h} $ as the test function in the variational formulation \eqref{error-phase-field.eq}, we have 
		\begin{align}
			\nonumber
			[ \CA _2 (\CJ _{h, \varphi}\varphi ^ {k+1}   , \BJ _{h, \Bu}\Bu ^ {k} ) &-\CA _2 (\varphi ^ {k+1}   , \Bu ^ {k} ) : e ^ {k+1}  _ {\varphi,h} ]
				\\& 
			\nonumber= [ \CA _2 (\varphi ^ {k+1}    , \BJ _{h, \Bu}\Bu ^ {k} ) -\CA _2 (\varphi  ^ {k+1}   , \Bu ^k) : e^ {k+1}  _ {\varphi,h} ]\\&\quad \nonumber
			+[ \CA _2 (\CJ _{h, \varphi}\varphi ^ {k+1}    , \BJ _ {h, \Bu}\Bu ^k) -\CA _2 ( \varphi  ^ {k+1}  ,  \BJ _ {h, \Bu}\Bu  ^k) : e ^ {k+1} _ {\varphi,h} ]\\\label{Proof-eq 2.58} 
			& := \widehat{T}_5+ \widehat{ T}_6.
		\end{align}
		Applying  the mean value theorem  analogous to \eqref{T1-A2}, we have
		\begin{align}
			\nonumber
			\widehat{ T}_5&= \skp{  \CA(\CE(\BJ _{h, \Bu}\Bu^k)):	(\CE(\BJ _{h, \Bu}\Bu ^k)-\CE(\Bu ^k)),\varphi ^ {k+1}\,e ^ {k+1} _ {\varphi,h} }\\&
			\nonumber
			\quad+\skp{\CA(\CE(\BJ _{h, \Bu}\Bu^k))-\CA (\CE (u^k)):	(\CE(\BJ _{h, \Bu}\Bu ^k)-\CE(\Bu ^k)), \varphi ^ {k+1}\,e ^ {k+1} _ {\varphi,h} }\\&
			:= \widehat{T}_ {5,1}+ \widehat{T}_ {5,2}.
		\end{align}		
Employing	the H\"older inequality \cite[Prop. II.2.18]{MR2986590},   the Sobolev embedding  $W^ {1,4} (\Omega ) \subset L^\infty (\Omega)$,  and noting that $\BJ _{h, \Bu}\Bu^k$ and $e^{k+1} _ {\varphi,h}$ are  piecewise linear functions as well as the approximation properties of $\BJ_{h, \Bu}$  and Young's inequality, we arrive at
\begin{align}
\nonumber
\widehat{T}_ {5,1} &\le \norm{\CA(\CE(\BJ _{h, \Bu}\Bu^k))}_ {\BL^ \infty(\Omega)}\norm{\CE(\BJ _{h, \Bu}\Bu^k)-\CE(\Bu^k)}_ {\BL^ 2(\Omega)}\norm{e ^ {k+1}_ {\varphi ,h}} _ {L^\infty (\Omega)}\norm{\varphi^{k+1}}_{L^2(\Omega)}\\[-3mm]&
\label{Proof-eq 2.49abcdef}
\le Ch^2\ell ^ {-1}\norm{\CA(\CE(\BJ _{h, \Bu}\Bu^k))}^2_ {\BL^ 2(\Omega)}\norm{\CE(\Bu^k)}^2_ {\BH^ 1(\Omega)}+\frac{\ell}{8 }\norm{e ^ {k+1}_ {\varphi ,h}}^2_ {H^1 (\Omega)}.
\end{align}
Similar
 to the previous estimate, in addition to applying the Lipschitz continuity of $\CA$ from Lemma \ref{Lem-holder contin}, and using the Sobolev embedding  $W^ {1,4} (\Omega ) \subset L^\infty (\Omega)$, and considering the fact that for  $T \in \T_h$,  ${e ^ {k+1}_ {\varphi ,h} |}_T \in P_1 (T)$, and  the inequality \eqref{Gaglias}, it follows that	
{ \begin{align}
\nonumber
\widehat{T}_ {5,2} &\le \norm{\left( \CA(\CE(\Bu^k))-\CA(\CE(\BJ _{h, \Bu}\Bu^k))\right) :\CE(\BJ _{h, \Bu}\Bu^k)-\CE(\Bu^k)}_ {\BL^ 2(\Omega)}\norm{\varphi^{k+1}}_{L^2(\Omega)}\norm{e ^ {k+1}_ {\varphi ,h}} _ {L^\infty (\Omega)}\\[2mm]&
\nonumber
\le C h \norm{\varphi^{k+1}}_{L^2(\Omega)} \norm{\CE(\Bu^k)}_ {\BH^ 1(\Omega)}\sum_{T \in \T_h}\norm{e ^ {k+1}_ {\varphi ,h}} _ {L^4 ( T)}\\[-4mm]
& \le Ch^2 \ell ^ {-1}  \norm{\varphi^{k+1}}_{L^2(\Omega)}  ^2\norm{\CE(\Bu^k)}^2_ {\BH^ 1(\Omega)}+ \frac{\ell}{8}\norm{e ^ {k+1}_ {\varphi ,h}}^2 _ {H^1 (\Omega)}.
 \end{align}}
The H\"older inequality \cite[Prop. II.2.18]{MR2986590},   the Sobolev embedding  $W^ {1,4} (\Omega ) \subset L^\infty (\Omega)$, and noting that $\BJ _{h, \Bu}\Bu^k$ is a piecewise linear function,  the Gagliardo–Nirenberg inequality \eqref{Gaglias}, as well as using the approximation properties of  $\CJ_{ h, \varphi}$ imply
{	\begin{align}
	\nonumber
	\widehat{ T}_6& \le \sum_{T \in \T_h}\norm{\CA(\CE(\BJ _{h, \Bu}\Bu^k))}_ {\BL^4( T)}\sum_{T \in \T_h}\norm{\CE(\BJ _{h, \Bu}\Bu^k)}_ {\BL^ d( T)} \norm{ \varphi ^ {k+1}- \CJ _h \varphi ^ {k+1} }_ {L^2 (\Omega)}\norm{e ^ {k+1}_ {\varphi ,h}} _ {L^2 (\Omega)}\\
	\label{Proof-eq 2.61}
	&\le Ch^2\ell\norm{\CA(\CE(\BJ _{h, \Bu}\Bu^k))}^2_ {\BL^ 2(\Omega)}\norm{\CE(\BJ _{h, \Bu}\Bu^k)}^2_ {\BL^ 2(\Omega)}  \norm{\varphi  ^ {k+1}}^2_ {H^1 (\Omega)}+ \frac{1}{8 \ell}\norm{e ^ {k+1}_ {\varphi ,h}}^2_ {L^2 (\Omega)},
	\end{align}} 
		Combining \eqref{Proof-eq 2.58}--\eqref{Proof-eq 2.61} results in 
		\begin{align}
		\nonumber
			\abs{[ \CA _2 (\CJ _{h,\varphi}\varphi  ^ {k+1} \right.& \left. , \BJ _{h, \Bu}\Bu^ {k} ) -\CA _2 (\varphi ^ {k+1}  , \Bu ^k) : e^ {k+1} _ {\varphi,h} ]}  
\\& \nonumber\le
 Ch^2\ell ^ {-1}\norm{\CA(\CE(\BJ _{h, \Bu}\Bu^k))}^2_ {\BL^ 2(\Omega)}\norm{\CE(\Bu^k)}^2_ {\BH^ 1(\Omega)}\\& \nonumber
 \quad+Ch^2 \ell ^ {-1}  \norm{\varphi^{k+1}}_{L^2(\Omega)}^2\norm{\CE(\Bu^k)}^2_ {\BH^ 1(\Omega)}\\
\nonumber
 &\quad+ Ch^2\ell\norm{\CA(\CE(\BJ _{h, \Bu}\Bu^k))}^2_ {\BL^ 2(\Omega)}\norm{\CE(\BJ _{h, \Bu}\Bu^k)}^2_ {\BL^ 2(\Omega)}  \norm{\varphi ^ {k+1}}^2_ {H^1 (\Omega)}\\
 \label{Proof-eq 2.60a}
 &\quad
+\frac{1}{8\ell }\norm{e ^ {k+1}_ {\varphi ,h}}^2 _ {L^2 (\Omega)}+\frac{\ell}{4}\norm{e ^ {k+1}_ {\varphi ,h}}^2 _ {H^1 (\Omega)}. 
		\end{align}	
		{\bf Step 2.3:}
		In this step, after using $w_h= e^ {k+1}_ {\varphi,h}$ as test function in \eqref{error-phase-field.eq}, we deal with the rest of the terms on the right hand side of this equation. Considering the definition of the Ritz operator $\CJ_{ h, \varphi}$, the first term on the right hand side of this equation vanishes, and for the second term following Young's inequality and  the approximation properties of this Ritz operator, we get 
	\begin{align}
	\label{Proof-eq 2.60am}
\widehat T _7:=\frac{1}{\ell} \skp{ {\varphi ^{k+1}-\CJ_{h, \varphi} \varphi ^{k+1} }, e^ {k+1}_ {\varphi,h}}\le \frac{Ch^2}{\ell}\norm{\varphi ^ {k+1}}^2_ {H^1 (\Omega)}+\frac{1}{8\ell}\norm{e ^ {k+1}_ {\varphi ,h}}^2 _ {L^2 (\Omega)},
\end{align}	
where $C$ is a nonegative constant independent of $\ell$, $\kappa$, $h$, and $\tau$.
\\
{{\bf Step 2.4: } Defining    $w _h:=e  ^ {k+1}_ {\varphi,h} $ as the test function in the variational formulation \eqref{error-phase-field.eq},  and using \eqref{eq.po} we conclude
\begin{align}\label{step.2.4.1}
	\nonumber
{\gamma_0}&\skp{[\varphi_h^{k+1}-\varphi_h ^{k}]_+-[\CJ_{h, \varphi}\varphi^{k}-\CJ_{h, \varphi}\varphi ^{k+1}]_+,e  ^ {k+1}_ {\varphi,h}  }\\&
\nonumber= {\gamma_0}\norm{[\varphi_h^{k+1}-\varphi_h ^{k}]_+-[\CJ_{h, \varphi}\varphi^{k}-\CJ_{h, \varphi}\varphi ^{k+1}]_+}_ {L^2 (\Omega)}^2\\&
+{\gamma_0}\skp{[\varphi_h^{k+1}-\varphi_h ^{k}]_+-[\CJ_{h, \varphi}\varphi^{k}-\CJ_{h, \varphi}\varphi ^{k+1}]_+,e  ^ {k}_ {\varphi,h}  } :=\widehat{T}_8+\widehat{T}_9.
\end{align}
For the second term in the right and side of the above equation, we can deduce from Young's inequality that
\begin{align}\label{step.2.4.2}
\abs{\widehat{T}_9 }\le \frac{\gamma_0}{4}\norm{[\varphi_h^{k+1}-\varphi_h ^{k}]_+-[\CJ_{h, \varphi}\varphi^{k}-\CJ_{h, \varphi}\varphi ^{k+1}]_+}^2_ {L^2 (\Omega)}+\norm{e  ^ {k}_ {\varphi,h} }^2_ {L^2 (\Omega)}.
\end{align}
Finally, applying \eqref{eq.po2}, following Young's inequality and  the approximation property of the Ritz operator $\CJ_{h, \varphi}$  lead to
\begin{align}
	\nonumber
{\gamma_0}&{\skp{[\varphi^{k+1}-\varphi ^{k}]_+-[ \CJ_{h, \varphi}\varphi^{k+1}-\CJ_{h, \varphi}\varphi ^{k}]_+ ,e  ^ {k+1}_ {\varphi,h}}}\le \gamma_0\norm{\varphi^{k+1}-\CJ_{h, \varphi}\varphi ^{k+1}}_ {L^2 (\Omega)}\norm{e ^ {k+1}_ {\varphi ,h}} _ {L^2 (\Omega)}\\
\label{penalty1}
&+\gamma_0\norm{\varphi^{k}-\CJ_{h, \varphi}\varphi ^{k}}_ {L^2 (\Omega)}\norm{e ^ {k+1}_ {\varphi ,h}} _ {L^2 (\Omega)}\le \frac{Ch^2}{\ell}\left( \norm{\varphi ^ {k}}^2_ {H^1 (\Omega)}+\norm{\varphi ^ {k+1}}^2_ {H^1 (\Omega)}\right) +\frac{1}{8\ell}\norm{e ^ {k+1}_ {\varphi ,h}}^2 _ {L^2 (\Omega)},
\end{align}}
where $C$ is a nonegative constant independent of $\ell$, $\kappa$, $h$, and $\tau$.

	{\bf Step 3:} \textbf{The heat error equation}. In this step, we consider the equation \eqref{error-heat.eq}, to control and simplify the nonlinear terms of this equation.\\
	{\bf Step 3.1:}
	We consider the left hand side of \eqref{error-heat.eq}, and start with substituting  $z_h= e^ {k+1}_ {\vartheta,h}$ as the test function in  \eqref{error-heat.eq}. Then exploiting  the definition of the Ritz operator $\CJ_{h, \vartheta}$ results in
	\begin{align}
	\nonumber 
\left(  K(\vartheta _h^ {k+1})\,\nabla \vartheta_h ^ {k+1}\right. & \left. -K(\CJ_h\vartheta ^ {k+1})\,\nabla \CJ_{h , \vartheta}\vartheta ^ {k+1}, \nabla  e^ {k+1}_ {\vartheta,h}\right)\\& \nonumber\ge c_0 \norm{\nabla e^ {k+1}_ {\vartheta,h}}^2 _ {\BL^2 (\Omega)}+c_0 \norm{(\abs{\vartheta _h}^ {\beta/2})\nabla e^ {k+1}_ {\vartheta,h}}^2 _ {\BL^2 (\Omega)}\\&
\nonumber
\quad+c_0\left( \left(  K(\vartheta _h^ {k+1}) - K(\CJ_h\vartheta ^ {k+1})\right)  \,\nabla \CJ_{h , \vartheta} \vartheta ^ {k+1}, \nabla  e^ {k+1}_ {\vartheta,h}\right)\\&
\label{Proof-eq 2.60ab}
:=H_1+H_2+H_3.
	\end{align}	
Indeed, thanks to the Lipschitz continuity of $K$  from Assumption \ref{assumption-initial data},	applying the H\"older inequality \cite[Prop. II.2.18]{MR2986590}, Cauchy-Schwarz inequality, and the Sobolev embedding $H^1 (\Omega ) \subset L^6 (\Omega) \subset  L^3 (\Omega) $ we conclude
	\begin{align}
\nonumber
	H_3 &\le C_{\mathcal{L}}\left( \norm{\vartheta _h^ {k+1}-\CJ_{h , \vartheta}  \vartheta ^ {k+1}}_{L^2(\Omega)}^2\right) \sum_{T \in \T_h}\norm{\nabla \CJ _{h , \vartheta}  \vartheta ^ {k+1}}^2_ {\BH^1 (T)} + \frac{c_0}{4}\sum_{T \in \T_h}\norm{\nabla e^ {k+1}_ {\vartheta,h}}^2 _ {\BH^1 ( T)}, 
	\end{align}
where $C_{\mathcal{L}}:=c_0\, C_{Lip}$ and $C_{Lip}$	is the Lipschitz continuity constant of $K$.
We   note that  $\CJ _{h , \vartheta}  \vartheta ^ {k+1}$ and $ e^ {k+1}_ {\vartheta,h}$ both belong to the space of  linear piecewise continuous functions, it is then obvious that  $\abs{\nabla \CJ _{h , \vartheta}  \vartheta ^ {k+1}}_{H^1(T)}=0$ and $\abs{\nabla e^ {k+1}_ {\vartheta,h}}_{H^1(T)}=0$ for all $T \in \T _h$. Hence, from this argument, combined with the a priori estimates \eqref{eq-a- priori estimate-semi discrete-heat} and \eqref{eq-a- priori estimate-full discrete-heat} we obtain that
\begin{align}
\nonumber 
	H_3 &\le C_{\mathcal{L}} \left( \norm {   {\vartheta _h^ {k+1}}}^{2}_ {L^2 (\Omega)}+\widehat C_ {st,\vartheta}\norm {   {\vartheta ^ {k+1}}}^{2}_ {L^2 (\Omega)}\right) \norm{\nabla \CJ _{h, \vartheta} \vartheta ^ {k+1}}^2_ {\BL^2 (\Omega)} + \frac{c_0}{4}\norm{\nabla e^ {k+1}_ {\vartheta,h}}^2 _ {\BL^2 (\Omega)}\\[2mm]
	\nonumber
	&\le C_{\mathcal{L}} \widehat C_ {st,\vartheta}\left( \norm {   {\vartheta _h^ {k+1}}}^{2}_ {L^2 (\Omega)}+\widehat C_ {st,\vartheta}\norm {   {\vartheta ^ {k+1}}}^{2}_ {L^2 (\Omega)}\right) \norm{\nabla  \vartheta ^ {k+1}}^2_ {\BL^2 (\Omega)} + \frac{c_0}{4}\norm{\nabla e^ {k+1}_ {\vartheta,h}}^2 _ {\BL^2 (\Omega)}\\
	\label{Proof-eq 2.49abcd} 
	&\le \widehat C_{\mathcal{L}} C_ {st,\vartheta}\left(1+\widehat C_ {st,\vartheta} \right)  \mathcal{L}_{1,k}  \tau \norm{\nabla  \vartheta ^ {k+1}}^2_ {\BL^2 (\Omega)}+ \frac{c_0}{4}\norm{\nabla e^ {k+1}_ {\vartheta,h}}^2 _ {\BL^2 (\Omega)},
\end{align}
where $\widehat C_{st,\vartheta}$ is the stability constant of $\CJ_{h, \vartheta }$.
		\\
	{\bf Step 3.2:} Here, we consider the last term in the left hand side of \eqref{error-heat.eq}, and set  $z_h= e^ {k+1}_ {\vartheta,h}$ to get
	\begin{align}
	\nonumber
	 [\CA _3 (\vartheta _h ^ k, \delta_\tau ^ k \Bu_h)-\CA _3 (\vartheta ^k, \delta_\tau ^ k \Bu): e^ {k+1}_ {\vartheta,h}]&= [\CA _3 (e^ {k}_ {\vartheta,h}, \delta_\tau ^ k \Bu_h): e^ {k+1}_ {\vartheta,h}]	 \\ & \nonumber \quad+[\CA _3 (\vartheta _h ^ {k}, \delta_\tau ^ k \Be_{\Bu,h}): e^ {k+1}_ {\vartheta,h}]
	 \\ & \nonumber \quad
	  +[\CA _3 (\CJ _{h , \vartheta}  \vartheta ^ {k}- \vartheta ^ {k}, \delta_\tau ^ k \CJ_{h , \Bu}\Bu): e^ {k+1}_ {\vartheta,h}]	 \\ & \nonumber \quad+[\CA _3 \left(  \vartheta ^k, \delta_\tau ^ k \left(\BJ_{h,\Bu} \Bu^k -\Bu^k \right) \right) : e^ {k+1}_ {\vartheta,h}]
	 \\
	 \label{Proof. main theorem-eq-3.29}
	 &{:=H_4+H_5+H_6+H_7.}
	\end{align}
By  the H\"older inequality \cite[Prop. II.2.18]{MR2986590},   the Sobolev embedding  $W^ {1,4} (\Omega ) \subset L^\infty (\Omega)$, Korn’s inequality \cite{horgan1983inequalities}, the  inequality \eqref{Gaglias}, and exploiting  the a priori estimate from Lemma \ref{Lem-a- priori estimate-discrete solution} and applying Young's inequality, we have 
	\begin{align}
	\nonumber
H_4 &\le C_K \rho \norm{e ^ {k+1}_ {\vartheta ,h}} _ {L^2 (\Omega)}\sum_{T \in \T_h}\norm{ \CE (\delta_\tau ^ k \Bu_h)} _ {\BL^4 ( T)} \norm{e ^ {k}_ {\vartheta ,h}} _ {L^2 (\Omega)}\\& \nonumber=  C_K \rho  \norm{e ^ {k+1}_ {\vartheta ,h}}^2 _ {L^2 (\Omega)}\norm{ \CE (\delta_\tau ^ k \Bu_h)} _ {\BL^2 (\Omega)} \\
\label{Proof-eq 2.49abcdefg}
	&
	\le { C_K \rho \sqrt{C^ \prime  _ {st, \CE\Bu}}} \left( \frac{\rho}{\sqrt{ \alpha _ \kappa} } \sqrt{\mathcal{L}_{1,L}}+ \frac{1 }{\sqrt{\alpha _{\kappa}}}
	\sum_{m=0}^{k}\norm{\Bf^m}_{\BL ^2 (\Omega)}\right) \norm{e ^ {k+1}_ {\vartheta ,h}} ^2_ {L^2 (\Omega)},
	\end{align}
	where the second estimate holds true since $\delta _\tau ^ k \Bu _h$ is a piecewise linear function. 
Completely analogous to the previous estimate, 
 the following upper bound holds true  for $H_5$
	\begin{align}
	\nonumber 
H_5 &\le  C_K \norm{\vartheta _h ^ {k}} _ {L^2 (\Omega)}\norm{ \CE (\delta_\tau ^ k \Be_{\Bu, h})} _ {\BL^\infty (\Omega)} \norm{e ^ {k+1}_ {\vartheta ,h}} _ {L^2 (\Omega)}\\
\label{Proof-eq 2.49abcdefgh}
	& \le  2C^2_K C_ {st, \vartheta} \tau ^2\mathcal{L}_{1,k}  \norm{ \CE (\delta_\tau ^ k \Be_{\Bu, h})}^2 _ {\BL^2 (\Omega)}+ \frac{1}{8 \tau } \norm{e ^ {k+1}_ {\vartheta ,h}} ^2_ {L^2 (\Omega)}, 
	\end{align}
{	where the last inequality obtained from the a priori estimate for $\vartheta _h^ k$ in Lemma \ref{Lem-a- priori estimate-discrete solution}.
		One can use  the H\"older inequality \cite[Prop. II.2.18]{MR2986590},  the Sobolev embedding  $W^ {1,4} (\Omega ) \subset L^\infty (\Omega)$, and note  that $\BJ_{h, \Bu} \Bu$ is a piecewise linear function, and apply the inequality \eqref{Gaglias} to get
\begin{align}
\nonumber
H_6 &\le \norm{\CJ _{h , \vartheta}  \vartheta ^ {k}- \vartheta ^ {k}}_{L^2(\Omega)}\sum_{T \in \T_h} \norm{\delta_\tau ^ k \CJ_{h , \Bu}\Bu}_{\BL^4(T)}  \norm{e ^ {k+1}_ {\vartheta ,h}}_ {L^2 (\Omega)}\\
\label{Proof-eq 2.50}
&\le C h^2 \tau \norm{\delta_\tau ^ k \Bu}^2_{\BH^1(\Omega)}\norm{\vartheta^k}^2_{H^1(\Omega)}+ \frac{1}{8 \tau } \norm{e ^ {k+1}_ {\vartheta ,h}} ^2_ {L^2 (\Omega)},
\end{align}
where the last term is a result of  applying the approximation property of $\CJ_{h , \vartheta}$ and Young’s inequality.}
Finally, from the H\"older inequality \cite[Prop. II.2.18]{MR2986590},  the Sobolev embedding $H^1 (\Omega ) \subset L^6 (\Omega) \subset  L^3 (\Omega) $ combined with	the approximation property of $\BJ_{h, \Bu}$ yield
\begin{align}
\nonumber
H_7&\le \norm{\vartheta^k}_{L^3(\Omega)}\norm{\delta_\tau ^ k \nabla \cdot \left(\BJ_{h,\Bu} \Bu^k -\Bu^k \right) }_{\BL^2(\Omega)}\norm{e^{k+1}_ {\vartheta,h}}_{L^6(\Omega)}\\
& \label{Proof-eq 2.50a}
\le C h^2 \norm{\vartheta^k}^2_{H^1(\Omega)}\norm{\delta_\tau^k \Bu}^2_{\BH^2(\Omega)}+\frac{c_0}{8}\norm{e^{k+1}_ {\vartheta,h}}^2_{H^1(\Omega)}.
\end{align} 
{\bf Step 3.3:} We set  $z_h= e^ {k+1}_ {\vartheta,h}$ in \eqref{error-heat.eq}, and   find upper bounds for the terms in the right hand side of this equation.
Using the approximation property of  $\CI_h$, there holds the following estimate
\begin{align}
H_8:=\skp{{\delta _\tau ^k(\vartheta- \CJ_{h, \vartheta } \vartheta)}, \, e ^ {k+1}_ {\vartheta ,h}}\le C h^2 \tau \norm{\delta _\tau ^k\vartheta}^2_{H ^1 (\Omega)}
+\frac{1 }{8 \tau  }\norm{e^ {k+1} _ {\vartheta ,h}}^2 _ {L^2 (\Omega)},
\end{align}
and 
\begin{align}
H_9:=\skp{\CI _h \gamma^ {k+1} - \gamma^ {k+1},e ^ {k+1}_ {\vartheta ,h}}&
\label{Proof-eq 2.63an} 
\le C h^2 \tau \norm{\gamma^ {k+1}}^2_{H ^1 (\Omega)}
+\frac{1 }{8 \tau  }\norm{e^ {k+1} _ {\vartheta ,h}}^2 _ {L^2 (\Omega)}.
\end{align}
{Finally, from the definition of the Ritz operator $\CJ_{h , \vartheta}$,  the Sobolev embedding  $W^ {1,4} (\Omega ) \subset L^\infty (\Omega)$, and  since $ e^ {k+1}_ {\vartheta,h}$  belongs to the space of  linear piecewise continuous functions,    the estimate \eqref{Gaglias}, as well as the approximation property of this operator, we conclude}
\begin{align}
\nonumber
H_{10}&:= \skp{\left( K(\vartheta ^ {k+1})\,\nabla \vartheta ^ {k+1}-K(\CJ_{h, \vartheta}\vartheta ^ {k+1})\,\nabla \CJ_{h,\vartheta}\vartheta ^ {k+1}  \right),\, \nabla e^ {k+1} _ {\vartheta ,h}}\\
\nonumber
&=\skp{ K(\CJ_{h,\vartheta} \vartheta ^ {k+1})\left( \nabla \CJ_{h,\vartheta}\vartheta ^ {k+1} -\nabla \vartheta ^ {k+1}\right)  ,\, \nabla e^ {k+1} _ {\vartheta ,h}}\\
& \nonumber \quad +\skp{\left( K(\vartheta ^ {k+1})-K(\CJ_{h, \vartheta}\vartheta ^ {k+1})  \right)\,\nabla \vartheta ^ {k+1},\, \nabla e^ {k+1} _ {\vartheta ,h}}\\
\label{Proof-eq 2.63a} 
&\le C c_0 ^{-1} h^2 \norm{\vartheta^ {k+1}}^2_{H^1 (\Omega)}+ \frac{c_0}{8}\norm{\nabla e^ {k+1} _ {\vartheta ,h}}_ {\BL^2(\Omega)}^2.
	\end{align}
{\bf Step 4:} \textbf{Collecting everything}.	
  		We use the test functions $\Bv _h=\delta_\tau ^k\Be _ {\Bu,h} $, $w _h=e^ {k+1} _ {\varphi,h} $, and  $z_h= e^ {k+1}_ {\vartheta,h}$ in the variational formulations \eqref{error-elasticity.eq}, \eqref{error-phase-field.eq} and \eqref{error-heat.eq}, respectively.
  	Then, we apply \eqref{upper bound-A1}, \eqref{Proof-eq 2.41}, \eqref{Proof-eq 2.48a}, \eqref{Proof-main-step3-eq1}, \eqref{Proof-eq 2.47}, { \eqref{step.2.4.1} }, \eqref{Proof-eq 2.60ab} and \eqref{Proof. main theorem-eq-3.29} to have the following inequality
		\begin{align}
			\nonumber 
		\frac{1}{2 \tau} \left(\norm {\partial _ \tau ^ {k} \Be _ {\Bu,h}  }^2_ {\BL^2 (\Omega)}\right.&  -\left. \norm {\partial _ \tau ^ {k-1}\Be _ {\Bu,h}}^2_ {\BL^2 (\Omega)}\right)+{\ell  }\norm{\nabla e ^ {k+1} _ {\varphi ,h}}^2 _ {\BL^2 (\Omega)}+\frac{1 }{\ell }\norm{e ^ {k+1}_ {\varphi ,h}}^2 _ {L^2 (\Omega)}
	 \\		 \nonumber
	&+	\frac{1}{\tau} \left(\norm {\Be ^ {k+1}_ {\vartheta,h}  }^2_ {L^2 (\Omega)}\right. -\left. \norm {\Be ^ {k-1}_ {\vartheta,h} }^2_ {L^2 (\Omega)}\right)+	c_0\norm{(\abs{\vartheta _h}^ {\beta/2})\nabla e^ {k+1}_ {\vartheta,h}}^2 _ {\BL^2 (\Omega)}
%
%
		 \\		 \nonumber
		 &
+c_0\norm{\nabla e^ {k+1}_ {\vartheta,h}}^2 _ {\BL^2 (\Omega)}
	+T_1+T_4+\widehat T_1 +\widehat{T}_8
	\\
	\label{main THM-proof-3.31}		
\le  &\sum_{i=2}^{11}\abs{T_i}+\sum_{i=2}^{7}\widehat T_i	+\widehat{T}_9+\sum_{i=3}^{10}H_i.
		\end{align}
After multiplying  both sides of \eqref{main THM-proof-3.31} into $\tau$ and considering 
	the    assumptions of this   theorem and  combining  \eqref{Proof-eq 2.41}, \eqref{Proof-eq 2.43}, \eqref{Proof-eq 2.44}, \eqref{proof-main-step3}--\eqref{Proof-eq 2.63},  \eqref{Proof-eq 2.48}, \eqref{Proof-eq 2.49a}--\eqref{Proof-eq 2.49m}, \eqref{Proof-eq 2.60a}, \eqref{Proof-eq 2.60am}, {\eqref{step.2.4.2}, \eqref{penalty1}}
	 \eqref{Proof-eq 2.49abcd}, \eqref{Proof-eq 2.49abcdefg}--\eqref{Proof-eq 2.63a},   applying the discrete Gronwall's lemma, and  for sufficiently small $h$ and $\tau$,  the following inequality holds true for all $k \le L\le M-1$  
		\begin{align}
		\nonumber
				\norm {\partial _ \tau ^ {L} \Be _ {\Bu,h}  }^2_ {\BL^2 (\Omega)}&+\kappa \tau  \norm{ \CE(\Be^ {L+1} _ {\Bu,h})}^2_ {\BL ^2 (\Omega)}+\norm {\Be ^ {L+1}_ {\vartheta,h}  }^2_ {L^2 (\Omega)}+\tau \norm {\nabla \Be ^ {L+1}_ {\vartheta,h}  }^2_ {\BL^2 (\Omega)} \\
				\nonumber
				& 
				+{\ell \tau }\norm{\nabla e ^ {L+1} _ {\varphi ,h}}^2 _ {\BL^2 (\Omega)} +\frac{\tau }{\ell }\norm{e ^ {L+1}_ {\varphi ,h}}^2 _ {L^2 (\Omega)} \lesssim  \mathcal{L}_{1,L+1}   \tau^2 \norm{\nabla  \vartheta ^ {L+1}}^2_ {\BL^2 (\Omega)}\\
				\nonumber
				&+ \kappa ^ {-1}h^2
		\left( \norm{ \CE (\Bu^{L+1}  )}^2_ {\BH ^1 (\Omega)} \norm{\varphi ^{L+1}} ^2 _ {H^1 (\Omega)}	
	+\norm{\vartheta ^{L}} ^2 _ {H^1 (\Omega)}\right)\\& \nonumber
		 +h^2  \norm{\CE(\Bu^{L+1})}^2_ {\BH^ 1(\Omega)}\norm{\varphi ^{L+1}} ^2 _ {H^1 (\Omega)}+
		 \ell ^ {-1}\tau h^2\left( \norm{ \CE (\Bu^{L+1}  )}^2_ {\BH ^1 (\Omega)} \norm{\varphi ^{L+1}} ^2 _ {H^1 (\Omega)}\right. \\
		 & \nonumber 
		 \left. +\norm{\varphi ^{L+1}} ^2 _ {H^1 (\Omega)}+\norm{ \CE (\Bu^{L+1}  )}^4_ {\BH ^1 (\Omega)} \right) +\ell\tau  h^2  \norm{\CE(\Bu^{L+1})}^4_ {\BL^2(\Omega)}\norm{\varphi ^{L+1}} ^2 _ {H^1 (\Omega)}\\
		 \nonumber
		 &
		 +\tau h^2\left(\norm{\partial_{\tau \tau} ^L \Bu}^2_ {\BH^1(\Omega)}+  \norm{\Bf^ {L+1}}^2_{\BH ^1 (\Omega)}+\norm{\vartheta^{L+1}}^2_ {H^1(\Omega)}\right)\\& 
		  \nonumber
		 +h^2 \tau^2 \left(\norm{\vartheta^{L}}^2_ {H^1(\Omega)} \norm{\delta ^L_{\tau }  \Bu}^2_ {\BH^2(\Omega)}+\norm{\delta ^L_{\tau } \vartheta}^2_{H^1(\Omega)}+\norm{\gamma ^ {L+1}}^2_{H^1(\Omega)}\right).   
		\end{align}
			Then, combining this with the triangle inequality and making use of the approximation properties of $\CJ_{ h,\varphi}$ , $\CJ_{h,\vartheta}$  and $\BJ _{h,\Bu}$  complete the proof.
	\end{proof}
\begin{remark}
In some practical examples, we need to assume   $\kappa=\mathcal{O}(\ell)$ and $h=\mathcal{O}(\ell)$, 
then in the statement of Theorem \ref{Thm. the main result}, we are required to add the assumptions ${\tau }{\ell ^ {-1}}= \mathcal{O}(1)$ and ${\tau }{\kappa ^ {-1}}= \mathcal{O}(1)$.
\end{remark}

\section{Numerical experiment}\label{numerical}
Here, we present a numerical example to illustrate the theoretical results.
We consider an area inside a square with a  length  of $1 \,\mathrm{mm}$ ($\Omega=(0,1)^2 \mathrm{mm}^2$) having a notch with a length of $0.5\;\mathrm{mm}$ (and a thickness of $1\,\mu \mathrm{m}$) on the left side as the domain. The time interval is considered to be  $I=[0,\,0.2].$
The specimen is
fixed at the bottom and we denote traction-free conditions on both sides. 
\\ 
A non-homogeneous Dirichlet condition is applied at the top.
In order to observe the material failure, we impose  a monotonic displacement $\overline{\Bu}=(0, 1\times10^{-5})^T$ at the top side in a
vertical direction (until the full fracture).
We also assume a zero load term  $\boldsymbol{f}=(0,0)^T$, zero initial displacement $\Bu _0=(0,0)^T$ and we  set $\Bv _0 =(0,0)^T$, $\gamma=0$, and {$\vartheta_0=0$}.
For the material parameters, we use a shear modulus of $\mu=13.33\times 10^{9}\,{Pa}$, a Lam{\'e} constant of $\lambda=8.88\times 10^{9}\,{Pa}$.
In this  problem, the stabilization parameter is assumed to be $\kappa = 10^{-8}$. The length scale is assumed  $\ell = \mathcal{O}(h)$, i.e., we set  $\ell = 2{h}$. Moreover,  the  energy release rate is   $G_c=3.0\times 10^{6}\,{Pa}$. For the temporal discretization, we use a time-step of $\tau=1\times 10^{-3}\,\text{s}$. A schematic of the computational domain is given in Figure \ref{schematic}. Regarding the thermal effect, we utilize a constant thermal conductivity $K=0.158$ W\hspace{-0.3mm}/\hspace{-0.3mm}m\,K, a Neumann boundary condition $\bar{\gamma}=\,300$K is imposed to the front of the notch (shown in $z$ in Figure \ref{schematic}), and the thermal expansion is $2\times 10^{-6}$\,$\mathrm{K^{-1}}$.

In order to solve the nonlinear system resulting from  \eqref{variational-elasticity-fully discrete.eq}-\eqref{variational-heat-fully discrete.eq}, we use Newton method's with 
the stopping criterion  $\texttt{Tol}_\texttt{N-R}=10^{-8}$, i.e. the relative
residual norm that is less than $\texttt{Tol}_\texttt{N-R}=10^{-8}$. At each 
Newton iteration, we use
a direct solver to solve the linear systems.
For this example, since there is no exact value for the displacement coordinates (i.e., $\textbf{u}_x,~\textbf{u}_y$), the function $\varphi$, and the temperature $\vartheta$ a reference observation employing 214\,321 elements and 214\,728 nodes is used to compute the error terms. The crack pattern and the heat distribution for this problem at the final  time step, i.e, $M=200$ are shown in Figure \ref{schematic}.
 \begin{figure} 
 	\centering
 	\includegraphics[width=0.9\textwidth]{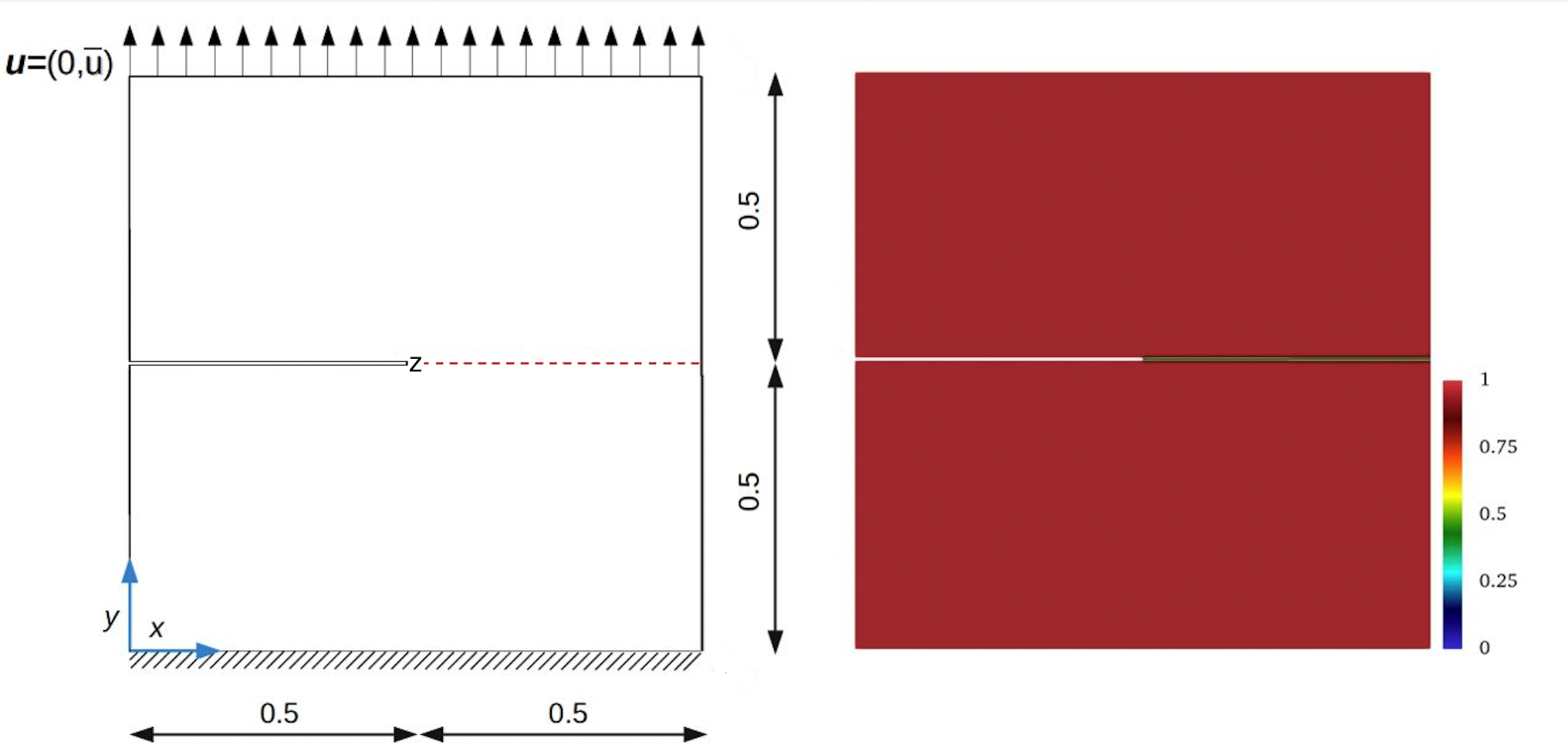}
 	\vspace{-0.25cm}
 	\caption{A Schematic of the single edge notch including its dimensions and boundary conditions (left) and the phase field $\varphi$ at the last time step (full failure) (right).}
 	\label{schematic}
 \end{figure}

 \begin{table}\label{table4}
 	\centering
 	\begin{tabular}{|l|cccccc|}
 		\hline
 		&  \quad\hspace{0.2cm} $||\textbf{u}^M-\textbf{u}^M_h||_{{\BL^2(\Omega)}}$\qquad          &   rate &     ~~  $||\varphi^M-\varphi^M_h||_{{L^2(\Omega)}}$  &     rate    &     ~~  $||\vartheta^M-\vartheta^M_h||_{{L^2(\Omega)}}$  &  ~~\,~~   rate    \\[0.8mm]\hline
 		$h=1/20$     & 0.0014   & --  & 0.2081  & --  & 0.102  & --\\[1mm]
 		$h=1/40$      & 8.10$\times 10^{-4}$   & 0.807  & 0.122  & 0.778 & 0.064  & 0.9075  \\[1mm]
 		$h=1/80$     & 4.25$\times 10^{-4}$   & 0.930  & 0.061  & 0.997  & 0.033  & 0.9556\\[1mm]				
 		$h=1/160$   &  2.07$\times 10^{-4}$  & 1.003  & 0.030& 1.002 & 0.0166  & 0.9911 \\[1mm]
 		$h=1/320$   &  1.03$\times 10^{-4}$   & 0.9958  & 0.015& 0.988 & 0.0082  & 1.0087 \\
 		\hline
 	\end{tabular}\\[2mm]
 	\centering
 	\begin{tabular}{|l|cccccc|}
 		\hline
 		& $|\nabla\left(\textbf{u}^M-\textbf{u}^M_h\right) |_{{\BL^2(\Omega)}}$\quad        &   rate &        $|\nabla\left(\varphi^M-\varphi^M_h\right) |_{{\BL^2(\Omega)}}$  &     rate    &        $|\nabla\left(\vartheta^M-\vartheta^M_h\right) |_{{\BL^2(\Omega)}}$  &     ~~rate~~    \\[0.5mm]\hline
 		$h=1/20$     & 0.0495   & --  & 6.90  & -- & 3.103  & -- \\[1mm]
 		$h=1/40$      &0.0403   & 0.2010  & 6.25  & 0.141& 2.261  & 0.457   \\[1mm]
 		$h=1/80$     & 0.0298   & 0.4355  & 4.61  & 0.442& 1.581  & 0.516  \\[1mm]				
 		$h=1/160$   &  0.0217  & 0.4576  & 3.39& 0.443& 1.110  & 0.509  \\[1mm]
 		$h=1/320$   &  0.0161   & 0.4390  & 2.46& 0.458& 0.785  & 0.498  \\
 		\hline
 	\end{tabular}\\[1.8mm]
 	\caption{The rate of  convergence in $L^2$-norm and $H^1$-semi norm for the discrete solutions of the test problem at the final time step, i.e., $M=200$.\vspace{-8mm}} 
 \end{table}

For the space discretization, we use  first-order quadrilateral finite elements for four integration points. In all time steps, the solutions are computed at the node; however, the derivatives are computed at the integration points. For the $L_2$-norm computations (of the derivatives), we interpolate from the Gauss points to the nodes. For this, at each point, we find the four closest integration points in the neighboring elements, estimate the weights with respect to the distances, and compute the derivative values. Table \ref{table4} shows the convergence of the error for solutions and the derivatives. Both results confirm the theoretical results.

%
%
%


\bibliographystyle{amsalpha}
\bibliography{references}
\end{document}